\documentclass[11pt]{article}
\usepackage{amssymb}
\usepackage{amsthm}
\usepackage{amsmath}
\usepackage{fullpage,hyperref,algorithmic}
\usepackage{graphicx,url,amsmath,amsthm,amsfonts,amssymb,subfigure,bbm}

\newtheorem{theorem}{Theorem}[section]
\newtheorem{lemma}[theorem]{Lemma}

\newtheorem{corollary}[theorem]{Corollary}

\newtheorem{observation}[theorem]{Observation}

\newcommand{\mc}{\mathcal}
\newcommand{\Var}{\mathrm{Var}}
\newcommand{\pr}{\mathbb{P}}

\newcommand{\jnote}[1]{}

\newcommand{\E}{{\mathbb E}}

\newcommand{\Lip}{\mathrm{Lip}}

\newcommand{\dist}{\mathsf{dist}}

\newcommand{\remove}[1]{}

\newcommand{\len}{\mathsf{len}}

\newcommand{\e}{\varepsilon}
\newcommand{\eps}{\varepsilon}

\newcommand{\col}{{\chi}}

\newcommand{\f}{{f}}

\newcommand{\g}{{g}}

\newcommand{\ceiling }[1]{{\left\lceil{#1}\right\rceil}}

\begin{document}

\date{}

\title{{\bf Dimension reduction for finite trees in $\ell_1$}}
\author{James R. Lee\footnote{Partially
supported by NSF grants CCF-0644037, CCF-0915251, and a Sloan Research Fellowship.  A significant
portion of this work was completed during a visit
of the authors to the Institut Henri Poincar\'e.} \\ University of Washington \and Arnaud de Mesmay\footnotemark[1] \\
Ecole Normale Sup\'erieure\and Mohammad Moharrami\footnotemark[1] \\ University of Washington}

\maketitle

\begin{abstract}
We show that every $n$-point tree metric admits
a $(1+\varepsilon)$-embedding into $\ell_1^{C(\varepsilon) \log n}$,
for every $\varepsilon > 0$, where $C(\varepsilon) \leq O\left((\frac{1}{\varepsilon})^4 \log \frac{1}{\varepsilon})\right)$.
This matches the natural volume lower bound up to a factor depending only on $\e$.
Previously, it was unknown whether
even complete binary trees on $n$ nodes could be embedded in $\ell_1^{O(\log n)}$ with $O(1)$ distortion.
For complete $d$-ary trees, our construction achieves $C(\varepsilon) \leq O\left(\frac{1}{\varepsilon^2}\right)$.
\end{abstract}

\setcounter{tocdepth}{2} \tableofcontents

\newpage

\section{Introduction}

Let $T=(V,E)$ be a finite, connected, undirected tree, equipped with a length function on edges, $\len : E \to [0,\infty)$.
This induces a shortest-path pseudometric\footnote{This is a pseudometric because we may have $d(u,v)=0$ even for distinct $u,v \in V$.},
$$
d_T(u,v) = \textrm{length of the shortest $u$-$v$ path in $T$}.
$$
Such a metric space $(V,d_T)$ is called a {\em finite tree metric.}

Given two metric spaces $(X,d_X)$ and $(Y,d_Y)$, and a mapping $f : X \to Y$, we define
the {\em Lipschitz constant of $f$} by,
$$
\|f\|_{\Lip} = \sup_{x \neq y \in X} \frac{d_Y(f(x),f(y))}{d_X(x,y)}.
$$
An {\em $L$-Lipschitz} map is one for which $\|f\|_{\Lip} \leq L$.
One defines the {\em distortion} of the mapping $f$ to be
$\dist(f) = \|f\|_{\Lip} \cdot \|f^{-1}\|_{\Lip}$, where
the distortion is understood to be infinite when $f$ is not injective.
We say that $(X,d_X)$ $D$-embeds into $(Y,d_Y)$ if there is a mapping $f : X \to Y$
with $\dist(f) \leq D$.

Using the notation $\ell_1^k$ for the space $\mathbb R^k$ equipped
with the $\|\cdot\|_1$ norm, we study the following question:
How large must $k=k(n,\varepsilon)$ be so that every $n$-point tree metric
$(1+\varepsilon)$-embeds into $\ell_1^k$?

\subsection{Dimension reduction in $\ell_1$}

A seminal result of Johnson and Lindenstrauss \cite{JL84}
implies that for every $\e > 0$, every $n$-point subset $X \subseteq \ell_2$
admits a $(1+\e)$-distortion embedding into $\ell_2^k$, with $k = O(\frac{\log n}{\e^2})$.
On the other hand, the known upper bounds for $\ell_1$ are much weaker.
Talagrand \cite{Tal90}, following earlier results of Bourgain-Lindenstrauss-Milman \cite{BLM89}
and Schechtman \cite{Sch87}, showed that every $n$-dimensional subspace $X \subseteq \ell_1$
(and, in particular, every $n$-point subset)
admits a $(1+\e)$-embedding into $\ell_1^k$, with $k=O(\frac{n \log n}{\e^2})$.
For $n$-point subsets, this was very recently improved to $k=O(n/\e^2)$ by Newman and Rabinovich \cite{NR10},
using the spectral sparsification techniques of Batson, Spielman, and Srivastava \cite{BSS09}.

On the other hand, Brinkman and Charikar \cite{BC05} showed that there exist
$n$-point subsets $X \subseteq \ell_1$ such that any $D$-embedding of $X$ into $\ell_1^k$
requires $k \geq n^{\Omega(1/D^2)}$ (see also \cite{LN-diamond} for a simpler proof).
Thus the exponential
dimension reduction achievable in the $\ell_2$ case cannot be matched
for the $\ell_1$ norm.
More recently, it has been show by Andoni, Charikar, Neiman, and Nguyen \cite{ACNN11} that
there exist $n$-point subsets such that any $(1+\e)$-embedding requires dimension
at least $n^{1-O(1/\log(\e^{-1}))}$.
Regev \cite{Regev11} has given an elegant proof of both
these lower bounds based on information theoretic arguments.

One can still ask about the possibility of more substantial dimension reduction
for certain finite subsets of $\ell_1$.  Such a study was undertaken by
Charikar and Sahai \cite{CS02}.
In particular, it is an elementary exercise to verify that every finite tree metric
embeds isometrically into $\ell_1$, thus the $\ell_1$ dimension reduction question
for trees becomes a prominent example of this type.
It was shown\footnote{The original
bound proved in \cite{CS02} grew like $\log^3 n$, but this was improved using an observation of A. Gupta.} \cite{CS02} that for every $\e > 0$,
every $n$-point tree metric $(1+\e)$-embeds into $\ell_1^k$ with $k = O(\frac{\log^2 n}{\e^2})$.
It is quite natural to ask whether the dependence on $n$ can be reduced to the natural volume
lower bound of $\Omega(\log n)$.  Indeed, it is Question 3.6 in the list ``Open problems
on embeddings of finite metric spaces'' maintained by J. Matou{\v{s}}ek \cite{MatOpen},
asked by Gupta, Lee, and Talwar\footnote{Asked at the DIMACS Workshop on Discrete Metric spaces and their Algorithmic Applications (2003).
The question
was certainly known to others before 2003, and was asked to the first-named author by Assaf Naor earlier that year.}.
As noted there, the question was, surprisingly, even open for
the complete binary tree on $n$ vertices.
The present paper resolves this question, achieving the volume lower bound for all finite trees.

\begin{theorem}\label{thm:main}
For every $\e > 0$ and $n \in \mathbb N$, the following holds.
Every $n$-point tree metric admits a $(1+\e)$-embedding
into $\ell_1^k$ with $k = O((\frac{1}{\e})^4 \log \frac{1}{\e} \log n)$.
\end{theorem}

The proof is presented in Section \ref{sec:colors}.
We remark that the proof also yields a randomized polynomial-time algorithm to construct the embedding.

\subsection{Notation}

For a graph $G=(V,E)$, we use the notations $V(G)$ and $E(G)$ to denote
the vertex and edge sets of $G$, respectively.
For a connected, rooted tree $T=(V,E)$ and $x,y \in V$, we use the notation $P_{xy}$
for the unique path between $x$ and $y$ in $T$, and $P_{x}$ for $P_{rx}$, where $r$
is the root of $T$.

For $k \in \mathbb N$, we write $[k] = \{1,2,\ldots,k\}$.
We also use the asymptotic notation $A \lesssim B$ to denote that $A = O(B)$,
and $A \asymp B$ to denote the conjunction of $A \lesssim B$ and $B \lesssim A$.

\subsection{Proof outline and related work}

\remove{First, our argument is inductive and probabilistic;
the structure of this induction is partially inspired
by the construction of Schulman for tree codes \cite{Schulman96}.
Secondly, we use the notion of caterpillar partitions
to control the topological complexity of a general finite tree.
These were also used by Matou{\v{s}}ek \cite{Matousek99}, and by \cite{GKL03,LNP09} (where
they are called ``monotone colorings''), to
embed trees into Hilbert spaces, and in \cite{CS02} for previous
low-dimensional embeddings of trees into $\ell_1$.
}

We first discuss the form that all our embeddings
will take.  Let $T=(V,E)$ be a finite, connected tree,
and fix a root $r \in V$.  For each $v \in V$, recall that
$P_v$ denotes the unique simple path from $r$ to $v$.
Given a labeling of edges by vectors $\lambda : E \to \mathbb R^k$,
we can define $\varphi : V \to \mathbb R^k$ by,
\begin{equation}\label{eq:canonical}
\varphi(x) = \sum_{e \in E(P_v)} \lambda(e).
\end{equation}
The difficulty now lies in choosing an appropriate labeling $\lambda$.
An easy observation is that if we have $\|\lambda(e)\|_1 = \len(e)$ for all $e \in E$
and the set $\{\lambda(e)\}_{e \in E}$ is orthogonal, then $\varphi$ is an isometry.
Of course, our goal is to use many fewer than $|E|$ dimensions for the embedding.
We next illustrate a major probabilistic technique employed in our approach.

\medskip
\noindent
{\bf Re-randomization.}  Consider an unweighted, complete binary tree of height $h$.
Denote the tree by $T_h = (V_h, E_h)$, let $n=2^{h+1}-1$ be the number of vertices, and let $r$
denote the root of the tree.
Let $\kappa \in \mathbb N$ be some constant which we will choose momentarily.
If we assign to every edge $e \in E_h$, a label $\lambda(e) \in \mathbb R^{\kappa}$, then there is
a natural mapping $\tau_{\lambda} : V_h \to \{0,1\}^{\kappa h}$ given by
\begin{equation}\label{eq:CBT}
\tau_{\lambda}(v) = (\lambda(e_1), \lambda(e_2), \ldots, \lambda(e_k), 0, 0, \ldots, 0),
\end{equation}
where $E(P_v) = \{e_1, e_2, \ldots, e_k\},$ and the edges are labeled
in order from the root to $v$.
Note that the preceding definition
falls into the framework of \eqref{eq:canonical},
by extending each $\lambda(e)$ to a $(\kappa h)$-dimensional vector padded with zeros,
but the specification here will be easier to work with presently.

If we choose the label map $\lambda : E_h \to \{0,1\}^{\kappa}$ uniformly at random, the probability for the embedding $\tau_{\lambda}$ specified in \eqref{eq:CBT} to have $O(1)$ distortion is at most exponentially small in $n$.
In fact, the probability for $\tau_{\lambda}$ to be injective is already this small.
This is because for two nodes $u,v \in V_h$ which are the children of the same node
$w$, there is $\Omega(1)$ probability that $\tau_{\lambda}(u)=\tau_{\lambda}(v)$, and there are $\Omega(n)$
such independent events.  In
Section \ref{sec:warmup}, we show that a judicious application of
the Lov\'asz Local Lemma \cite{EL75} can be
used to show that $\tau_{\lambda}$ has $O(1)$ distortion with non-zero probability.
In fact, we show that this approach can handle arbitrary $k$-ary
complete trees, with distortion $1+\e$.
Unknown to us at the time of discovery, a closely related
construction occurs in the context of tree codes for interactive communication \cite{Schulman96}.

\medskip

Unfortunately, the use of the Local Lemma does not extend well to the 
more difficult setting of arbitrary trees.
For the general case,
we employ an idea of Schulman \cite{Schulman96} based on {\em re-randomization}.
To see the idea in our simple setting,
consider $T_h$ to be composed of a root $r$, under which lie two copies
of $T_{h-1}$, which we call $A$ and $B$, having roots $r_A$ and $r_B$, respectively.

The idea is to assume that, inductively,
we already have a
labeling $\lambda_{h-1} : E_{h-1} \to \{0,1\}^{\kappa (h-1)}$ such that
the corresponding map $\tau_{\lambda_{h-1}}$ has $O(1)$ distortion on $T_{h-1}$.
We will then construct a random labeling $\lambda_h : E_h \to \{0,1\}^{\kappa}$
by using $\lambda_{h-1}$ on the $A$-side, and $\pi(\lambda_{h-1})$ on the $B$-side,
where $\pi$ randomly alters the labeling in such a way that
$\tau_{\pi(\lambda_{h-1})}$ is simply $\tau_{\lambda_{h-1}}$ composed
with a random isometry of $\ell_1^{\kappa(h-1)}$.  We will then argue that
with positive probability (over the choice of $\pi$), $\tau_{\lambda_h}$ has $O(1)$ distortion,

\medskip

Let $\pi_1, \pi_2, \ldots, \pi_{h-1} : \{0,1\}^{\kappa} \to \{0,1\}^{\kappa}$ be
i.i.d. random mappings, where the distribution
of $\pi_1$ is specified by
$$
\pi_1(x_1, x_2, \ldots, x_{\kappa}) = \left(\rho_1(x_1), \rho_2(x_2), \ldots, \rho_{\kappa}(x_{\kappa})\right),
$$
where each $\rho_i$ is an independent uniformly random involution $\{0,1\} \mapsto \{0,1\}$.
To every edge $e \in E_{h-1}$, we can assign a height $\alpha(e) \in \{1,2,\ldots,h-1\}$ which
is its distance to the root.  From a labeling $\lambda : E_{h-1} \to \{0,1\}^{\kappa}$, we define a random labeling $\pi(\lambda) : E_{h-1} \to \{0,1\}^{\kappa}$ by,
$$
\pi(\lambda)(e) = \pi_{\alpha(e)} \circ \lambda\,.
$$
By a mild abuse of notation, we will consider $\pi(\lambda) : E(B) \to \{0,1\}^{\kappa}$.

Finally, given a labeling $\lambda_{h-1} : E_{h-1} \to \{0,1\}^{\kappa}$, we construct
a random labeling $\lambda_h : E_h \to \{0,1\}^{\kappa}$ as follows,
$$
\lambda_h(e) =
\begin{cases}
(0,0,\ldots,0) & e=(r,r_A) \\
(1,1,\ldots,1) & e=(r,r_B) \\
\lambda_{h-1}(e) & e \in E(A) \\
\pi(\lambda_{h-1})(e) & e \in E(B)\,.
\end{cases}
$$

By construction, the mappings $\tau_{\lambda_h}|_{V(A) \cup \{r\}}$ and $\tau_{\lambda_h}|_{V(B) \cup \{r\}}$
have the same distortion as $\tau_{\lambda_{h-1}}$.
In particular, it is easy to check that $\tau_{\pi(\lambda_{h-1})}$ is simply $\tau_{\lambda_{h-1}}$
composed with an isometry of $\{0,1\}^{\kappa (h-1)}$.

Now consider some pair $x \in V(A)$ and $y \in V(B)$.
It is simple to argue that it suffices
to bound the distortion for pairs with
$m=d_{T_h}(r,x)=d_{T_h}(r,y)$, for $m \in \{1,2,\ldots,h\}$,
so we will assume that $x,y$ have the same height in $T_h$.

Observe that $\tau_{\lambda_h}(x)$ is fixed with respect to the randomness in $\pi$,
thus if we write $v = \tau_{\lambda_h}(x) - \tau_{\lambda_h}(y)$, where subtraction
is taken coordinate-wise, modulo 2, then $v$ has the form
$$
v \equiv \left(\underbrace{1,1,\ldots,1}_{\kappa}, b_1, b_2, \ldots, b_{\kappa(m-1)}\right)\,
$$
where the $\{b_i\}$ are i.i.d. uniform over $\{0,1\}$.
It is thus an easy consequence of Chernoff bounds
that, with probability at least $1- e^{-m\kappa/8}$, we have
$$
\|\tau_{\lambda_h}(x)-\tau_{\lambda_h}(y)\|_1 = \|v\|_1 \geq \frac{\kappa \cdot d_{T_h}(x,y)}{4}\,.
$$
Also, clearly $\|\tau_{\lambda_h}\|_{\Lip} \leq \kappa$.

On the other hand, the number of pairs $x \in V(A), y \in V(B)$ with $m=d_{T_h}(r,x)=d_{T_h}(r,y)$ is
$2^{2(m-1)}$, thus taking a union bound, we have
$$
\pr\left(\dist(\tau_{\lambda_h}) > \max\{4, \dist(\tau_{\lambda_{h-1}})\}\right)
\leq \sum_{m=1}^{h} 2^{2(m-1)} e^{-m\kappa/8},
$$
and the latter bound is strictly less than 1 for some $\kappa = O(1)$,
showing the existence of a good map $\tau_{\lambda_h}$.

This illustrates how re-randomization (applying a distribution over random isometries to one side of a tree)
can be used to achieve $O(1)$ distortion for embedding $T_h$ into $\ell_1^{O(h)}$.  Unfortunately,
the arguments become significantly more delicate when we handle less uniform trees.
The full-blown re-randomization argument occurs in Section \ref{sec:embedding}.

\medskip
\noindent
{\bf Scale selection.}
The first step beyond complete binary trees would be in passing to complete $d$-ary trees for $d \geq 3$.
The same construction as above works, but now one has to choose $\kappa \asymp \log d$.
Unfortunately, if the degrees of our tree are not uniform, we have to adopt a significantly
more delicate
strategy.  It is natural to choose a single number $\kappa(e) \in \mathbb N$ for every edge $e \in E$,
and then put $\lambda(e) \in \frac{1}{\kappa(e)} \{0,1\}^{\kappa(e)}$ (this ensures that
the analogue of the embedding $\tau_{\lambda}$ specified in \eqref{eq:CBT} is 1-Lipschitz).

Observing the case of $d$-ary trees, one might be tempted to put
$$
\kappa(e) = \left\lceil \log \frac{|T_u|}{|T_v|}\right\rceil,
$$
where $e=(u,v)$ is directed away from the root, and we use $T_v$ to denote the subtree
rooted at $v$.  If one simply takes a complete binary tree on $2^h$ nodes, and then connects
a star of degree $2^h$ to every vertex, we have $\kappa(e) \asymp h$ for every edge,
and thus the dimension becomes $O(h^2)$ instead of $O(h)$ as desired.

In fact, there are examples which show that it is impossible to choose $\kappa(u,v)$ to depend only on
the geometry of the subtree rooted at $u$.  These ``scale selector'' values have to look at the global
geometry, and in particular
have to encode the volume growth of the tree at many scales simultaneously.
  Our eventual scale selector is fairly sophisticated and impossible to describe
without delving significantly into the details of the proof.
For our purposes,
we need to consider more general embeddings of type \eqref{eq:canonical}.  In particular,
the coordinates of our labels $\lambda(e) \in \mathbb R^k$ will take a range of different values,
not simply a single value as for complete trees.

We do try to maintain one important, related invariant:
If $P_v$ is the sequence of edges from the root to some vertex $v$,
then ideally for every coordinate $i \in \{1,2,\ldots,k\}$ and every value $j \in \mathbb Z$,
there will be at most one $e \in P_v$ for which $\lambda(e)_i \in [2^j, 2^{j+1})$.
Thus instead of every coordinate being ``touched'' at most once on the path from the root to $v$,
every coordinate is touched at most once {\em at every scale} along every such path.
This ensures that various scales do not interact.
For technical reasons,
this property is not maintained exactly, but analogous concepts arise frequently in the proof.

The restricted class of embeddings we use, along with a discussion of the
invariants we maintain, are introduced in Section \ref{sec:overscales}.
The actual scale selectors are defined in Section \ref{sec:assignment}.

\medskip
\noindent
{\bf Controlling the topology.}
One of the properties that we used above for complete $d$-ary trees is
that the depth of such a tree is $O(\log_d n)$, where $n$ is the number of nodes in the tree.
This allowed us to concatenate vectors down a root-leaf path without exceeding
our desired $O(\log n)$ dimension bound.  Of course, for general trees,
no similar property need hold.
However, there is still a bound on the {\em topological} depth of any $n$-node tree.

To explain this, let $T=(V,E)$ be a tree with root $r$, and define a {\em monotone coloring of $T$}
to be a mapping $\chi : E \to \mathbb N$ such that for every $c \in \mathbb N$, the
color class $\chi^{-1}(c)$ is a connected subset of some root-leaf path.
Such colorings were used in previous works on
embedding trees into Hilbert spaces \cite{Matousek99,GKL03,LNP09}, as well
as for preivous low-dimensional embeddings into $\ell_1$ \cite{CS02}.
The following
lemma is well-known and elementary.

\begin{lemma}\label{lem:monotone}
Every connected $n$-vertex rooted tree $T$ admits a monotone coloring
such that every root-leaf path in $T$ contains at most $1+\log_2 n$ colors.
\end{lemma}

\begin{proof}
For an edge $e \in E(T)$, let $\ell(e)$ denote the number of leaves beneath
$e$ in $T$ (including, possibly, an endpoint of $e$).
Letting $\ell(T) = \max_{e \in E} \ell(e)$, we will prove
that for $\ell(T) \geq 1$,
there exists a monotone coloring with at most $1+\log_2 (\ell(T)) \leq 1+\log_2 n$ colors
on any root-leaf path.

Suppose that $r$ is the root of $T$.  For an edge $e$, let $T_e$ be the subtree beneath $e$,
including the edge $e$ itself.  If $r$ is the endpoint of edges $e_1, e_2, \ldots, e_k$,
we may color the edges of
$T_{e_1}, T_{e_2}, \ldots, T_{e_k}$ separately, since
any monotone path is contained completely within exactly one of these subtrees.
Thus we may assume that $r$ is the endpoint
of only one edge $e_1$, and then $\ell(T)=\ell(e_1)$.

Choose a leaf $x$ in $T$ such that each connected component of $T'$ of $T \setminus E(P_{rx})$
has $\ell(T') \leq \ell(e_1)/2$ (this is easy to do by, e.g., ordering the leaves from left to right
in a planar drawing of $T$).
Color the edges $E(P_{rx})$ with color 1, and inductively
color each non-trivial connected component $T'$ with disjoint sets of
colors from $\mathbb N \setminus \{1\}$.
By induction, the maximum number of colors appearing on a root-leaf path in $T$
is at most $1+\log_2(\ell(e_1)/2) = 1+\log_2(\ell(T))$, completing the proof.
\end{proof}

Instead of dealing directly with edges in our actual embedding,
we will deal with color classes.  This poses a number of difficulties,
and one major difficulty involving vertices which occur in the middle
of such classes.  For dealing with these vertices,
we will first preprocess our tree by embedding it into a product
of a small number of new trees, each of which
admits colorings of a special type.  This is carried
out in Section \ref{sec:colors}.

\section{Warm-up:  Embedding complete $k$-ary trees}
\label{sec:warmup}

We first prove our main result for the special case of complete $k$-ary trees, with an improved
dependence on $\e$.  The main novelty is our use of the Lov\'asz Local Lemma to analyze
a simple random embedding of such trees into $\ell_1$.
The proof illustrates the tradeoff being concentration and the
sizes of the sets $\{ \{u,v\} \subseteq V : d_T(u,v)=j \}$ for each $j=1,2,\ldots$.

\begin{theorem}\label{thm:kary}
Let $T_{k,h}$ be the unweighted, complete $k$-ary tree of height $h$.
For every $\e > 0$, there exists a $(1+\e)$-embedding of $T_{k,h}$ into $\ell_1^{O((h \log k)/\e^2)}$.
\end{theorem}

In the next section, we introduce our random embedding and analyze the
success probability for a single pair of vertices based on their distance.
Then in Section \ref{sec:local}, we show that with non-zero probability,
the construction succeeds for all vertices.
In the coming sections and
later, in the proof of our main theorem,
we will employ the following concentration inequality \cite{McD98}.

\begin{theorem}\label{thm:azuma}
Let $M$ be a non-negative number, and $X_i~(1\leq i\leq n)$ be independent random variables satisfying $X_i\leq \E(X_i)+M$, for $1\leq i\leq n$.  Consider the sum $X=\sum_{i=1}^n X_i$ with expectation $\E(X)=\sum_{i=1}^n \E(X_i)$  and $\Var(X)=\sum_{i=1}^n \Var(X_i)$.
Then we have,
\begin{equation}\label{eq:azuma}
\pr(X- \mathbb{E}(X)\ge \lambda)\leq \exp\left({-\lambda^2\over 2(\Var(X)+M\lambda/3) }\right).
\end{equation}
\end{theorem}

\subsection{A single event}

First $k,h \in \mathbb N$ and $\e > 0$.  Write $T=(V,E)$ for the tree $T_{k,h}$ with root $r \in V$, and let $d_T$ be the unweighted
shortest-path metric on $T$.
Additionally, we define,
\begin{equation}\label{eq:def:kary:t}
t=\ceiling {1\over \eps},
\end{equation}
and
\begin{equation}\label{eq:def:kary:m}
m=t\ceiling{\log k}.
\end{equation}

Let $\{\vec v(1),\ldots, \vec v(t)\}$, be the standard basis for $\mathbb R^t$.
Let $b_1, b_2, \ldots, b_m$ be chosen i.i.d. uniformly over $\{1,2,\ldots,t\}$.
For the edges $e \in E$, we choose i.i.d. random labels $\lambda(e) \in \mathbb R^{m \times t}$,
each of which has the distribution of the random vector (represented in matrix notation),
\begin{equation}\label{eq:def:lambda}
{1\over m}
\left(
\begin{array}{c}
\vec v(b_1)\\
\vdots \\
\vec v(b_{m})
\end{array}
\right)\,.
\end{equation}

Note that for every $e\in E$, we have $\|\lambda(e)\|_1=1$.
We now define a random mapping $g:V\to \mathbb R^{m(h-1) \times t}$ as follows:
We put $g(r)=0$, and otherwise,
\begin{equation}\label{eq:def:g}
g(v)=\left(\begin{array}{c}
\lambda(e_1)\\
\vdots \\
 \lambda(e_{j})\\
 0\\
\vdots \\
0
\end{array}
\right),
\end{equation}
where $e_1, e_2, \ldots, e_{j}$ is the sequence
of edges encountered on the path from the root to $v$.
It is straightforward to check that $g$ is $1$-Lipschitz.
The next observation is also immediate from the definition of $g$.

\begin{observation}\label{lem:kary:rlpath}
For any $v\in V$ and $u\in V(P_v)$, we have
$d_T(u,v)= \|g(u)-g(v)\|_1$.
\end{observation}

For $m,n\in \mathbb N$, and $A\in \mathbb R^{m\times n}$, we use the notation $A[i] \in \mathbb R^n$ to refer to the $i$th row of $A$.
We now bound the probability that a given pair of vertices experiences a large contraction.

\begin{lemma}\label{lem:kary:con}
For $C\geq10$, and $x,y\in V$,
\begin{equation}\label{eq:kary:con:result}
\pr\left[\vphantom{\bigoplus} \|g(x)-g(y)\|_1\leq (1-C\eps)d_T(x,y)\right]\leq k^{-Cd_T(x,y)/2}\,.
\end{equation}
\end{lemma}

\begin{proof}
Fix $x,y \in V$, and let $r'$ denote their lowest common ancestor.
We define the family random variables $\{X_{ij}\}_{i\in [h-1], j \in [m]}$ by
setting $\ell_{ij} = (i-1)m + j$, and then
\begin{equation}\label{eq:Xij}
X_{ij}=\|g(x)[\ell_{ij}]-g(r')[\ell_{ij}]\|_1+\|g(y)[\ell_{ij}]-g(r')[\ell_{ij}]\|_1-\|g(x)[\ell_{ij}]-g(y)[\ell_{ij}]\|_1\,.
\end{equation}

Observe that if $i \leq d_T(r,r')$ then $X_{ij}=0$ for all $j \in [m]$ since all three terms in \eqref{eq:Xij} are zero.
Furthermore, if $i \geq \min(d_T(r,x), d_T(r,y))+1$, then again $X_{ij}=0$ for all $j \in [m]$, since in this case
one of the first two terms of \eqref{eq:Xij} is zero, and the other is equal to the last.  Thus if
$$R = [h-1] \cap [d_T(r,r')+1, \min(d_T(r,x),d_T(r,y))],$$ then
$i \notin R \implies X_{ij} = 0$ for all $j \in [m]$, and additionally we have
the estimate,
\begin{equation}\label{eq:sizeR}
|R| = \min(d_T(r,x),d_T(r,y))-d_T(r,r') \leq \frac{d_T(x,y)}{2}\,.
\end{equation}

Now, using the definition of $g$ \eqref{eq:def:g}, we can write
\begin{align*}
\|g(x)-g(y)\|_1&=\sum_{i \in [h-1], j \in [m]} \left(\|g(x)[\ell_{ij}]-g(r')[\ell_{ij}]\|_1+\|g(y)[\ell_{ij}]-g(r')[\ell_{ij}]\|_1-X_{ij}\right)\\
&=\|g(x)-g(r')\|_1+\|g(y)-g(r')\|_1-\sum_{i \in [h-1], j \in [m]} X_{ij}\\
&\overset{\eqref{lem:kary:rlpath}}{=}d_T(x,r')+d_T(y,r')-\sum_{i \in [h-1], j \in [m]} X_{ij}\\
&=d_T(x,y)-\sum_{i \in [h-1], j \in [m]} X_{ij}\,.
\end{align*}
We will prove the lemma by arguing that,
\[
\pr\left[\sum_{i \in [h-1], j \in [m]} X_{ij}\leq C\eps d_T(x,y)\right]\leq k^{-Cd_T(x,y)/2}.
\]

We start the proof by first bounding the maximum of the $X_{ij}$ variables. Since,
for every $\ell$, we have
$$\|g(x)[\ell]-g(r')[\ell]\|_1,\,\|g(y)[\ell]-g(r')[\ell]\|_1 \in \left\{0,\frac{1}{m}\right\},$$ we conclude that,
\begin{equation}\label{eq:kary:max}
\max\left\{\vphantom{\bigoplus} X_{ij}:i\in [h-1], j \in [m] \right\}\leq {2\over m}.
\end{equation}

For $i\in R$ and $j \in [m]$, using \eqref{eq:def:lambda} and \eqref{eq:def:g},
we see that
 $(g(x)[\ell_{ij}]-g(r')[\ell_{ij}])={1\over m}\vec v(\alpha)$ and $g(y)[\ell_{ij}]-g(r')[\ell_{ij}]={1\over m}\vec v(\beta)$, where $\alpha$ and $\beta$ are i.i.d. uniform over $\{1,\ldots, t\}$.
 Hence, for $i\in R$ and $j \in [m]$, we have $$\pr[X_{ij}\neq 0]={1\over t}\,.$$

 We can thus bound  the expected value and variance of $X_{ij}$ for $i\in R$ and $j \in [m]$ using \eqref{eq:kary:max},
\begin{equation}\label{eq:Xi:E}
\E[X_{ij}]\leq{2\over tm}\,,
\end{equation}
and
\begin{equation}\label{eq:Xi:Var}
\Var(X_{ij})\leq {4\over tm^2}\,.
\end{equation}

Using \eqref{eq:sizeR}, we have
\begin{align}\label{eq:Xsum:E}
\sum_{i=1}^{h-1} \sum_{j=1}^m \E[X_{ij}]&=\sum_{i\in R} \sum_{j \in [m]} \E[X_{ij}]
 \overset{\eqref{eq:Xi:E}}{\leq}\sum_{i\in R}{2\over t}
 \overset{\eqref{eq:sizeR}}{\leq} {d_T(x,y)\over t},
\end{align}
and
\begin{align}\label{eq:Xsum:Var}
\sum_{i=1}^{h-1} \sum_{j=1}^{m} \Var(X_{ij})&
=\sum_{i\in R} \sum_{j \in [m]} \Var(X_{ij})
 \overset{\eqref{eq:Xi:Var}}{\leq}\sum_{i\in R}{4\over tm}
 \overset{\eqref{eq:sizeR}}{\leq} {2\,d_T(x,y)\over tm}\,.
\end{align}
We now apply Theorem~\ref{thm:azuma} to complete the proof:
\begin{align*}
{\pr}\Bigg[\sum_{i \in [h-1],j \in [m]} X_{ij}&\geq C\left({d_T(x,y)\over t}\right)\Bigg]\\
&={\pr}\Bigg[\sum_{i \in [h-1], j \in [m]} X_{ij}-{d_T(x,y)\over t}\geq (C-1)\left({d_T(x,y)\over t}\right)\Bigg]\\
&\overset{\eqref{eq:Xsum:E}}{\leq}
{\pr}\left(\sum_{i \in [h-1], j \in [m]} {X_{ij} }-\mathbb E\left[\sum_{i \in [h-1], j\in [m]} X_{ij}\right] \geq (C-1)\left(d_T(x,y)\over t\right)\right)\\
&\leq\exp\left({-((C-1)d_T(x,y)/t)^2\over 2\left(\sum_{i \in [h-1], j \in [m]} \Var(X_{ij})+ (C-1)(d_T(x,y)/t) ({2\over m})/3\right)}\right)\\
&\overset{\eqref{eq:Xsum:Var}}\leq
\exp\left({-((C-1)d_T(x,y)/t)^2\over2\left({2\,d_T(x,y)/ (tm)}+ (C-1)(d_T(x,y)/t) ({2\over m})/3\right)}\right)\\
&=\exp\left({-(C-1)^2\over4\left(1+ (C-1)/3\right)}\cdot\frac{m}{t}\cdot d_T(x,y) \right).
\end{align*}
An elementary calculation shows that for $C\geq 10$,  we have
${(C-1)^2 \over {4(1+(C-1)/3)}}\geq {C\over 2}.$
Hence,
\begin{align*}
{\pr}\Bigg[\sum_{i \in [h-1],j \in [m]} X_{ij} \geq C\eps{d_T(x,y)}\Bigg]
&\overset{\eqref{eq:def:kary:t}}{\leq}{\pr}\Bigg[\sum_{i \in [h-1],j \in [m]} X_{ij} \geq C\left({d_T(x,y)\over t}\right)\Bigg]\\
&\leq\exp\left(-{Cm\over 2t}d_T(x,y)\right)\\
&\overset{\eqref{eq:def:kary:m}}{\leq} k^{-Cd_T(x,y)/2}\,
\end{align*}
completing the proof.
\end{proof}

\subsection{The Local Lemma argument}
\label{sec:local}

We first give the statement of the
Lov\'asz Local Lemma \cite{EL75} and then use it in conjunction with Lemma~\ref{lem:kary:con} to complete the proof
of Theorem \ref{thm:kary}.

\begin{theorem}
\label{thm:LLL}
Let $\mathcal A$ be a finite set of events in some probability space.
For $A \in \mathcal A$, let $\Gamma(A) \subseteq \mathcal A$ be such
that $A$ is independent from the collection of events $\mathcal A \setminus (\{A\} \cup \Gamma(A))$.
If there exists an assignment $x : \mathcal A \to (0,1)$ such that for all $A \in \mathcal A$, we have
$$
\pr(A) \leq x(A) \prod_{B \in \Gamma(A)} (1-x(B)),
$$
then the probability that none of the events in $\mathcal A$ occur is at least $\prod_{A \in \mathcal A} (1-x(A)) > 0$.
\end{theorem}

\begin{proof}[Proof of Theorem~\ref{thm:kary}]
We may assume that $k \geq 2$.
We will use Theorem~\ref{thm:LLL} and Lemma~\ref{lem:kary:con} to show that with non-zero probability the following inequality holds for all $u,v\in V$,
\[
\|g(u)-g(v)\|_1\leq (1-14\eps)\,d_T(u,v).
\]

For $u,v\in V$, let $\mathcal E_{uv}$, be the event $\left\{ \|g(u)-g(v)\|_1 \leq (1-{14\eps})\,d_T(u,v)\right\}$.
Now, for $u,v\in V$, define
$$x_{uv}=k^{-3 d_T(u,v)}\,.$$
Observe that for vertices $u,v \in V$ and a subset $V' \subseteq V$, the event $\mc E_{uv}$
is mutually independent of the family $\{ \mc E_{u'v'} : u',v' \in V' \}$ whenever the
induced subgraph of $T$ spanned by $V'$ contains no edges from $P_{uv}$.
Thus using Theorem~\ref{thm:LLL}, it is sufficient to show that for all $u,v\in V$,
\begin{equation}\label{eq:lll}
\pr(\mc E_{uv}) \leq x_{uv} \mathop{\prod_{s,t\in V :}}_{E(P_{st})\cap E(P_{uv})\neq \emptyset} (1-x_{st})\,.
\end{equation}
Indeed, this will complete the proof of Theorem \ref{thm:kary}.

To this end, fix $u,v \in V$.
For $e\in E$ and $i\in \mathbb N$, we define the set,
\[S_{e,i}=\{(u,v):\textrm{$u,v\in V$, $d_T(u,v)=i$, and $e \in E(P_{uv})$}\}.
\]
Since $T$ is a $k$-ary tree,
\begin{equation}\label{eq:s:size}
|S_{e,i}|\leq \sum_{j=1}^{i} k^{j-1}\cdot k^{i-j}= i\cdot k^{i-1}\leq k^{2i}.
\end{equation}
Thus we can write,
\begin{align*}
x_{uv}\mathop{\prod_{s,t\in V :}}_{E(P_{st})\cap E(P_{uv})\neq \emptyset} (1-x_{st})&=
x_{uv} \prod_{e \in E(P_{uv})}  \prod_{i\in \mathbb N}  \prod_{(s,t) \in S_{e,i}} \left(1- x_{st}\right)\\
&= k^{-3d_T(u,v)} \prod_{e \in E(P_{uv})}  \prod_{i\in \mathbb N}  \prod_{(s,t) \in S_{e,i}} \left(1- k^{- 3i}\right)\\
&\overset{\eqref{eq:s:size}}{\geq}  k^{-3d_T(u,v)} \prod_{e \in E(P_{uv})}  \prod_{i\in \mathbb N}  \left(1- k^{ - 3i}\right)^{k^{2i}}\\
&\geq k^{-3d_T(u,v)} \prod_{e \in E(P_{uv})}  \prod_{i\in \mathbb N}  \left(1- k^{2i}(k^{ - 3i})\right)\\
&= k^{-3d_T(u,v)} \prod_{e \in E(P_{uv})}  \prod_{i\in \mathbb N}  \left(1- \frac 1 {k^{i}}\right).
\end{align*}
For $x\in [0,\frac12]$, we have $e^{-2x}\leq 1-x$, and since $k \geq 2$, we have $k^{-i} \leq {1\over 2}$ for all $i\in \mathbb N$, hence
\begin{align*}
x_{uv}\mathop{\prod_{s,t\in V :}}_{E(P_{st})\cap E(P_{uv})\neq \emptyset} \left(1-x_{st}\right)
&\geq k^{-3d_T(u,v)} \prod_{e \in E(P_{uv})}  \prod_{i\in \mathbb N}  \exp\left({- 2\over k^{i}}\right)\\
&= k^{-3d_T(u,v)} \prod_{e \in E(P_{uv})}   \exp\left(- 2\sum_{i\in \mathbb N} {1\over k^i}\right)\\
&= k^{-3d_T(u,v)} \prod_{e \in E(P_{uv})}   \exp\left({- 2/k\over 1-1/k}\right)\\
&\geq k^{-3d_T(u,v)} \prod_{e \in E(P_{uv})}   \exp\left({- 4 \over k}\right)\\
&= k^{-3d_T(u,v)} \exp\left({- 4\,d_T(u,v) \over k}\right).
\end{align*}
Since $k\geq 2$, we conclude that,
\begin{equation*}
x_{uv}\mathop{\prod_{s,t\in V :}}_{E(P_{st})\cap E(P_{uv})\neq \emptyset} \left(1-x_{st}\right)
\geq k^{- 7d_T(u,v)}.
\end{equation*}
On the other hand, Lemma~\ref{lem:kary:con} applied with $C=14$ gives,
\[
\pr\left[\|g(u)-g(v)\|_1\leq (1-14\eps)d_T(u,v)\right]\leq k^{-7d_T(u,v)},
\]
yielding \eqref{eq:lll}, and completing the proof.

\end{proof}

\section{Colors and scales}

In the present section, we develop some tools for our eventual embedding.
The proof of our main theorem appears in the next section,
but relies on a key theorem which is only proved in Section \ref{sec:embedding}.

\subsection{Monotone colorings}
\label{sec:colors}

Let $T=(V,E)$ be a metric tree rooted at a vertex $r \in V$.
Recall that such a tree $T$ is equipped with a length $\len : E \to [0,\infty)$.
We extend this to subsets of edges $S \subseteq E$ via $\len(S) = \sum_{e \in S} \len(e)$.
We recall that a {\em monotone coloring} is a mapping $\chi : E \to \mathbb N$
such that each color class $\chi^{-1}(c) = \{ e \in E : \chi(e) = c \}$ is a connected subset
of some root-leaf path.
For a set of edges $S \subseteq E$, we write $\chi(S)$ for the
set of colors occurring in $S$.
We define the {\em multiplicity of $\chi$}
by
$$
M(\chi) = \max_{v \in V} |\chi(P_v)|\,.
$$

Given such a coloring $\chi$ and $c \in \mathbb N$,
we define, $$\len_{\chi}(c) = \len(\chi^{-1}(c)),$$
and $\len_{\chi}(S) = \sum_{c \in S} \len_{\chi}(c)$, if $S \subseteq \mathbb N$.

For every $\delta \in [0,1]$ and $x,y \in V$, we define
the set of colors
\begin{equation*}\label{eq:sigcolors}
C_{\chi}(x,y; \delta) = \left\{ c : \len(P_{xy} \cap \chi^{-1}(c)) \leq \delta \cdot \len_{\chi}(c) \right\} \cap (\chi(P_x) \triangle \chi(P_y))\,.
\end{equation*}
This is the set of colors $c$ which occur in only one of $P_x$ and $P_y$,
and for which the contribution to $P_{xy}$ is significantly smaller than $\len_{\chi}(c)$.
We also put,
\begin{equation}\label{eq:rho}
\rho_{\chi}(x,y;\delta) = \len_{\chi}(C(x,y;\delta))\,.
\end{equation}

We now state a key theorem that will be proved in Section \ref{sec:embedding}.

\begin{theorem}\label{thm:allbut}
For every $\varepsilon, \delta > 0$, there is a value $C(\varepsilon,\delta) = O((\frac{1}{\varepsilon} +\log \log \frac{1}{\delta})^3 \log \frac{1}{\e})$
such that the following holds.
For any metric tree $T=(V,E)$ and any monotone coloring $\chi : E \to \mathbb N$, there
exists a mapping $F : V \to \ell_1^{C(\varepsilon,\delta) (\log n + M(\chi))}$, such that
for all $x,y \in V$,
\begin{equation}\label{eq:allbut}
(1 - \e) \,d_T(x,y) - \delta\, \rho_{\chi}(x,y;\delta) \leq \|F(x)-F(y)\|_1 \leq d_T(x,y)\,.
\end{equation}
\end{theorem}

The problem one now confronts is whether the loss in the $\rho_{\chi}(x,y; \delta)$ term
can be tolerated.  In general, we do not have a way to do this,
so we first embed our tree into a product of a small number of trees
in a way that allows us to control the corresponding $\rho$-terms.

\begin{lemma}\label{lem:folding}
For every $\eps \in (0,1)$, there is a number $k \asymp \frac{1}{\e}$
such that the following holds.
For every metric tree $T=(V,E)$ and monotone coloring $\chi : E \to \mathbb N$,
there exist
$k$ metric trees $T_1, T_2, \ldots, T_k$ with monotone
colorings $\{\chi_i : E(T_i) \to \mathbb N\}_{i=1}^k$
and mappings $\{f_i : V \to V(T_i)\}_{i=1}^k$
such that $M(\chi_i) \leq M(\chi)$, and $|V(T_i)|\leq |V|$ for all $i \in [k]$, and
the following conditions hold for all $x,y \in V:$
\begin{enumerate}
\item[(a)] We have,
\begin{equation}\label{eq:lowerb}
\frac{1}{k} \sum_{i=1}^k d_{T_i}(f_i(x),f_i(y)) \geq (1-\eps)\,d_T(x,y)\,.
\end{equation}
\item[(b)] For all $i \in [k]$, we have
\begin{equation}\label{eq:foldlipschitz}
d_{T_i}(f_i(x),f_i(y)) \leq (1+\eps)\,d_T(x,y)\,.
\end{equation}
\item[(c)] There exists a number $j \in [k]$ such that
\begin{equation}\label{eq:rholoss}
\eps \,d_T(x,y)\geq \frac{2^{-(k+1)}}{k} \mathop{\sum_{i=1}^k}_{i \neq j} \rho_{\chi_i}(f_i(x),f_i(y);2^{-(k+1)})
\end{equation}
\end{enumerate}
\end{lemma}

Using Lemma \ref{lem:folding} in conjunction with Theorem \ref{thm:allbut}, we can now
prove the main theorem (Theorem \ref{thm:main}).

\begin{proof}[Proof of Theorem \ref{thm:main}]
Let $\e > 0$ be given,
let $T=(V,E)$ be an $n$-vertex metric tree.
Let $\chi : E \to \mathbb N$ be a monotone coloring
with $M(\chi) \leq O(\log n)$, which exists by Lemma \ref{lem:monotone}.
Apply Lemma \ref{lem:folding} to obtain
metric trees $T_1, \ldots, T_k$ with corresponding monotone colorings $\chi_1, \ldots, \chi_k$
and a mappings $f_i : V \to V(T_i)$.
Observe that $M(\chi_i) \leq O(\log n)$ for each $i \in [k]$.

Let $F_i : V(T_i) \to \ell_1^{C(\e) \log n}$ be the mapping
obtained by applying Theorem \ref{thm:allbut} to $T_i$ and $\chi_i$, for each $i \in [k]$,
with $\delta = 2^{-(k+1)}$, where $C(\e) = O(\frac{1}{\e^3} (\log \frac{1}{\e}))$.
Finally, we put $$F = \frac{1}{k} \left((F_1 \circ f_1) \oplus (F_2 \circ f_2) \oplus \cdots \oplus (F_k \circ f_k)\right)$$ so that
$F : V \to \ell^{O((\frac{1}{\e})^4 \log \frac{1}{\e} \cdot \log n)}$.
We will prove that $F$ is a $(1+O(\e))$-embedding, completing the proof.

\medskip

First, observe that each $F_i$ is $1$-Lipschitz (Theorem \ref{thm:allbut}).  In conjunction with condition~(b) of Lemma \ref{lem:folding}
which says that $\|f_i\|_{\Lip} \leq 1+\e$ for each $i \in [k]$, we have $\|F\|_{\Lip} \leq 1+\e$.

For the other side, fix $x,y \in V$ and let $j \in [k]$ be the number guaranteed in condition~(c) of Lemma \ref{lem:folding}.
Then we have,
\begin{eqnarray*}
\|F(x)-F(y)\|_1 &=& \frac{1}{k} \sum_{i=1}^k \|(F_i \circ f_i)(x) - (F_i \circ f_i)(y)\|_1 \\
& \overset{\eqref{eq:allbut}}{\geq} &
\frac{1}{k} \sum_{i \neq j} \left((1-\e)\,d_{T_i}(f_i(x),f_i(y)) - 2^{-(k+1)} \rho_{\chi_i}(f_i(x),f_i(y); 2^{-(k+1)})\right) \\
& \overset{\eqref{eq:rholoss}}{\geq} &
\left(\frac{1}{k} \sum_{i \neq j} (1-\e)\,d_{T_i}(f_i(x),f_i(y))\right) - \e\, d_T(x,y) \\
& \geq &
\left(\frac{1}{k} \sum_{i=1}^k (1-\e)\,d_{T_i}(f_i(x),f_i(y))\right) - \frac{1}{k} \, d_{T_j}(f_j(x), f_j(y)) - \e\, d_T(x,y)  \\
& \overset{\eqref{eq:foldlipschitz}}{\geq} &
\left(\frac{1}{k} \sum_{i=1}^k (1-\e)\,d_{T_i}(f_i(x),f_i(y))\right) - \frac{1+\e}{k} \, d_T(x,y) - \e\, d_T(x,y)  \\
& \overset{\eqref{eq:lowerb}}{\geq} &
(1-\e)^2 \,d_T(x,y) - \frac{1+\e}{k} \, d_T(x,y) - \e\, d_T(x,y) \\
& \geq &
(1-O(\e)) \,d_T(x,y)\,
\end{eqnarray*}
where in the final line we have used $k \asymp \frac{1}{\e}$, completing the proof.
\end{proof}

We now move on to the proof of Lemma \ref{lem:folding}.
We begin by proving an analogous statement for the half line $[0,\infty)$.
An {\em $\mathbb R$-star} is a metric space formed as follows:
Given a sequence $\{a_i\}_{i=1}^{\infty}$ of positive numbers,
one takes the disjoint union of the intervals $\{[0,a_1], [0,a_2], \ldots\}$,
and then identifies the 0 point in each, which is canonically called the {\em root of the $\mathbb R$-star.}
An $\mathbb R$-star $S$ carries the natural induced length metric $d_S$.  We refer
to the associated intervals as {\em branches}, and the {\em length of a branch} is the associated number $a_i$.
Finally, if $S$ is an $\mathbb R$-star, and $x \in S \setminus \{0\}$, we use $\ell(x)$ to denote
the length of the branch containing $x$.  We put $\ell(0)=0$.

\begin{lemma}\label{lem:fold:path}
For every $k \in \mathbb N$ with $k \geq 2$, there exist
$\mathbb R$-stars $S_1, \ldots, S_k$ with mappings
$$f_i : [0,\infty) \to S_i$$ such that the following conditions hold:
\begin{enumerate}
\item For each $i \in [k]$, $f_i(0)$ is the root of $S_i$.
\item For all $x,y \in [0,\infty)$, $\frac{1}{k} \sum_{i=1}^k d_{S_i}(f_i(x),f_i(y)) \geq \left(1-\frac{7}{k}\right) |x-y|\,.$
\item For each $i \in [k]$, $f_i$ is $(1+2^{-k+1})$-Lipschitz.
\item For $x \in [0,\infty)$, we have $\ell(f_i(x)) \leq 2^{k-1} x.$
\item For $x\in [0,\infty)$, there are at most two values of $i \in [k]$ such that
\[
d_{S_i}(f_i(0),f_i(x)) \leq  2^{-k} \,\ell(f_i(x))\,.
\]

\item For all $x,y \in [0,\infty)$, there is at most one value of $i \in [k]$ such that
$f_i(x)$ and $f_i(y)$ are in different branches of $S_i$ and
\[
2^{-k} \left(\ell(f_i(x)) + \ell(f_i(y))\right) \leq 2\, |x-y|\,.
\]
\end{enumerate}
\end{lemma}
\begin{proof}
Assume that $k \geq 2$.
We first construct $\mathbb R$-stars $S_1, \ldots, S_k$.
We will index the branches of each star by $\mathbb Z$.
For $i \in [k]$, $S_i$ is a star whose $j$th branch, for $j \in \mathbb Z$, has length $2^{i-1+k(j+1)}$.
We will use the notation $(i,j,d)$ to denote the point at distance $d$
from the root on the $j$th branch of $S_i$.
Observe that $(i,j,0)$ and $(i,j',0)$ describe
the same point (the root of $S_i$) for all $j,j' \in \mathbb N$.

Now, we define for every $i \in [k]$,
a function $f_i : [0,\infty) \to S_i$ as follows:
\[
f_i(x)=\left\{
\begin{array}{ll}
\big(i,j,(x-2^{i+kj})/(1-2^{1-k})\big) &\textrm{for $2^{-i} x\in [2^{kj},2^{k(j+1)-1})$},\\
\big(i,j,2^{i+k(j+1)}-x\big)&\textrm{for $2^{-i} x \in [2^{k(j+1)-1},2^{k(j+1)})$}.
\end{array}
\right.
\]
Condition~(i) is immediate.
It is also straightforward to verify that
\begin{equation}\label{eq:fi:lip}
\|f_i\|_{\Lip} \leq (1-2^{1-k})^{-1} \leq 1+2^{-k+1}\,
\end{equation}
yielding condition~(iii).

Toward verifying condition~(ii), observe that for every $x \in [0,\infty)$ and $j\in \{0,1, \ldots, k-2\}$ we have $$d_{S_i}(f_i(x),0)\geq {\left(x-2^{\lfloor\log_2 x\rfloor-j}\right)/ (1-2^{1-k})}\geq x-2^{\lfloor\log_2 x\rfloor-j},$$
 when $i = (\lfloor\log_2 x\rfloor-j) \bmod k$. Using this, we can write

\begin{align}\label{eq:fold0}
\sum_{i=1}^{k} d_{S_i}(f_i(x),f_i(0))&\geq
\sum_{j=\lfloor \log_2 x\rfloor -k+2}^{\lfloor \log_2 x \rfloor} x-2^j&\nonumber\\
&=
(k-1)x-\sum_{j=\lfloor \log_2 x \rfloor-k+2}^{\lfloor \log_2 x \rfloor} 2^{j}\nonumber\\
&\geq
(k-1)x-2^{\lfloor \log_2 x \rfloor+1}\nonumber \\&\geq (k-3)x.
\end{align}

Now fix $x,y \in [0,\infty)$ with $x \leq y$.
If $x\leq y/2$, then we can use the triangle inequality, together with \eqref{eq:fi:lip} and
\eqref{eq:fold0} to write,
\begin{eqnarray*}
\frac{1}{k} \sum_{i=1}^k d_{S_i}(f_i(x),f_i(y)) &\geq &
\frac{1}{k} \sum_{i=1}^k \left(\vphantom{\bigoplus} d_{S_i}(f_i(y),f_i(0)) - d_{S_i}(f_i(x),f_i(0))\right) \\
&\geq& (1-3/k)y-(1+2^{1-k})x\\&\geq& (1-3/k)y-(1+1/k)x\\& \geq& (1-7/k)(y-x)+ 4y/k-8x/k\\&\geq&(1-7/k)(y-x).
\end{eqnarray*}
In the case that ${y\over 2} \leq x\leq y$, for $j\in \{0,1,\ldots, k-3\}$, we have
$$d_{S_i}(f_i(x),f_i(y))\geq (y-x)/(1-2^{1-k})\geq y-x,$$
when $i=(\lfloor\log_2 x\rfloor-j) \bmod k$. From this, we conclude that
\begin{align}\label{eq:fold}
{1\over k}\sum_{i=1}^{k} d_{S_i}(f_i(x),f_i(y))&\geq
{1\over k}\sum_{j=0}^{k-3}(y-x)\geq{ k-2 \over k}(y-x),
\end{align}
yielding condition~(ii).

It is also straightforward to check that $$\ell(f_i(x))\leq2^{\lfloor \log_2 x\rfloor +k-1}\leq 2^{k-1}x,$$
which verifies condition~(iv).

To verify condition~(v), note that for $x\in [0,\infty)$, the inequality $d_{S_i}(f_i(x),f_i(0))\leq x/2$ can only hold for
$i \bmod k \in \{\lfloor \log_2 x\rfloor, \lfloor \log_2 x \rfloor + 1\}$, hence condition~(iv) implies condition~(v).

Finally we verify condition~(vi). We divide the problem into two cases. If $x<y/2$, then by condition~(iv),
\[
\ell(f_i(x)) + \ell(f_i(y)) \leq 2^{k-1} (x+y)\leq 2^{k-1} (2y)\leq 2^{k+1}(y-x)\,.
\]
In the case that $y/2< x\leq y$, $f_i(x)$ and $f_i(y)$ can be mapped to different branches of $S_i$ only for $i \equiv \lfloor \log_2 y\rfloor~(\bmod~k)$, yielding condition~(vi).
\end{proof}

Finally, we move onto the proof of Lemma \ref{lem:folding}.

\begin{proof}[Proof of Lemma~\ref{lem:folding}]
We put $k = \lceil 7/\varepsilon\rceil$ and
prove the following stronger statement by induction on $|V|$:
There exist metric trees $T_1, T_2, \ldots, T_k$ and monotone colorings
$\chi_i : E(T_i) \to \mathbb N$, along with mappings $f_i : V \to V(T_i)$
satisfying the conditions of the lemma.  Furthermore, each coloring $\chi_i$
satisfies the stronger condition for all $v \in V$,
\begin{equation}\label{eq:ind-cond}
|\chi_i(P_{f_i(v)})| \leq |\chi(P_v)|\,.
\end{equation}

The statement is trivial for the tree containing only a single vertex.
Now suppose that we have a tree $T$ and coloring $\chi : E \to \mathbb N$.
Since $T$ is connected,
it is easy to see that there exists a color class $c \in \chi(E)$ with
the following property.  Let $\gamma_c$ be the path whose edges are colored $c$,
and let $v_c$ be the vertex of $\gamma_c$ closest to the root.
Then the induced tree $T'$ on the vertex set $(V \setminus V(\gamma_c)) \cup \{v_c\}$
is connected.

Applying the inductive hypothesis to $T'$ and $\chi|_{E(T')}$ yields
metric trees $T_1', T_2', \ldots, T_k'$ with colorings $\chi_i' : E(T_i') \to \mathbb N$
and mappings $f'_i : V(T') \to V(T_i')$.

\medskip

Now, let $S_1, \ldots, S_k$ and $\{ g_i : [0,\infty) \to S_i\}$ be the
$\mathbb R$-stars and mappings guaranteed by Lemma \ref{lem:fold:path}.
For each $i \in [k]$,
let $S_i'$ be the induced subgraph of $S_i$ on the set $\{g_i(d_T(v,v_c)) : v \in V(\gamma_c)\}$,
and make $S_i'$ into a metric tree rooted at $g_i(0)$, with the length
structure inherited from $S_i$.
We now construct $T_i$ by attaching $S'_i$ to $T'_i$ with the root
of $S'_i$ identified with the node $f_i'(v_c)$.
The coloring $\chi_i'$ is extended to $T_i$ by assigning
to each root-leaf path in $S'_i$ a new color.
Finally, we specify functions $f_i : V \to V(T_i)$ via
$$
f_i(v) = \begin{cases}
f'_i(v) & v \in V(T') \\
g_i(d_T(v_c, v)) & v \in V \setminus V(T')\,.
\end{cases}
$$

It is straight forward to verify that \eqref{eq:ind-cond} holds for the colorings $\{\chi_i\}$ and every vertex $v \in V$.
In addition, using the inductive hypothesis, we have $|V(T_i)| \leq |V|$ and $M(\chi) \leq M(\chi_i)$ for every $i \in [k]$,
with the latter condition following immediately from \eqref{eq:ind-cond} and the structure of the mappings $\{f_i\}$.

We now verify that conditions (a), (b), and (c) hold.
For $x,y \in V(T')$, the induction hypothesis guarantees all three conditions.
If both $x,y \in V(\gamma_c)$, then
conditions (a) and (b) follow directly from conditions (ii) and (iii) of Lemma \ref{lem:fold:path}
applied to the maps $\{g_i\}$.  To verify condition (c), let $j \in [k]$ be the single bad
index from (vi).
We have for all $i\neq j$,
$$
 \rho_{\chi_i}(f_i(x), f_i(y); 2^{-(k+1)}) \leq  2^{k+1}d_T(x,y).
$$
Since there are at most two colors on the path between $x$ and $y$ in any $T_i$, by condition (v) of Lemma \ref{lem:fold:path}, there are at most four values of $i \in [k] \setminus \{j\}$ such that
$$\rho_{\chi_i}(f_i(x), f_i(y); 2^{-(k+1)})\neq 0,$$
hence
$${1\over k}\sum_{i \neq j} \rho_{\chi_i}(f_i(x), f_i(y); 2^{-(k+1)}) \leq  {4\cdot 2^{k+1}\over k}\,d_T(x,y)\leq \eps 2^{k+1}d_T(x,y).
$$

Since $\|f_i\|_{\Lip}$ is determined on edges $(x,y) \in E$, and each such edge has $x,y \in V(\gamma_c)$ or $x,y \in V(T')$,
we have already verified condition (b) for all $i \in [k]$ and $x,y \in V$.
Finally, we verify (a) and (c) for pairs with $x\in V(T')$ and $y\in V(\gamma_c)$.
We can check condition (a) using the previous two cases,
\begin{align*}
\frac{1}{k} \sum_{i=1}^k d_{T_i}(f_i(x),f_i(y))&=\frac{1}{k} \sum_{i=1}^k \left(\vphantom{\bigoplus} d_{T_i}(f_i(x),f_i(v_c))+d_{T_i}(f_i(y),f_i(v_c))\right)\\
&\geq (1-\eps)d_T(y,v_c)+(1-\eps)d_T(x,v_c)\\
&\geq (1-\eps)d_T(x,y).
\end{align*}

Towards verifying condition (c), note that by condition (v) from Lemma~\ref{lem:fold:path}, there are at most two values of $i$, such that
$$\rho_{\chi_i}(f_i(x),f_i(y);2^{-(k+1)})- \rho_{\chi_i}(f_i(x),f_i(v_c);2^{-(k+1)})=\rho_{\chi_i}(f_i(y),f_i(v_c);2^{-(k+1)})\neq 0.$$
By the induction hypothesis, there exists a number $j\in[k]$ such that
\[
\eps\, d_T(x,v_c)\leq \frac{2^{-(k+1)}}{k} \mathop{\sum}_{i \neq j} \rho_{\chi_i}(f_i(v_c),f_i(x);2^{-(k+1)}).
\]
Now we use condition (iv) from Lemma~\ref{lem:fold:path} to conclude,
\begin{align*}
\frac{2^{-(k+1)}}{k} \mathop{\sum}_{i \neq j} \rho_{\chi_i}(f_i(x),f_i(y);2^{-k})
&\leq \frac{2^{-(k+1)}}{k} \mathop{\sum}_{i \neq j}\left( \rho_{\chi_i}(f_i(x),f_i(v_c);2^{-k})+\rho_{\chi_i}(f_i(y),f_i(v_c);2^{-k})\right)
\\
&\leq \eps d_T(x,v_c)+\left(2^{-(k+1)}\over k\right)\,( 2^{k-1} d_T(y,v_c))\\
&\leq \eps \,d_T(x,v_c)+\eps \,d_T(v_c,y)\\
&= \eps \,d_T(x,y)\,,
\end{align*}
completing the proof.
\end{proof}

\subsection{Multi-scale embeddings}
\label{sec:overscales}

We now present the basics of our multi-scale embedding approach.
The next lemma is devoted to combining scales together without using too many dimensions,
while controlling the distortion of the resulting map.

\begin{lemma} \label{lem:scales}
For every $\e \in (0,1)$, the following holds.
Let $(X,d)$ be an arbitrary metric space, and
consider a family of functions $\{f_i  : X \to [0,1]\}_{i \in \mathbb Z}$
such that for all $x,y \in X$, we have
\begin{equation}\label{eq:converge}
\sum_{i \in \mathbb Z} 2^i |f_i(x)-f_i(y)| < \infty\,.
\end{equation}
Then there is a mapping $F : V \to \ell_1^{2+\lceil \log \frac{1}{\e}\rceil}$ such that
for all $x,y \in X$,
\[
(1-\eps)\sum_{i\in \mathbb Z}2^{i} |f_i(x)-f_i(y)|- 2\,\zeta(x,y)  \leq \|F(x)-F(y)\|_1\leq \sum_{i\in \mathbb Z} 2^{i}|f_i(x)-f_i(y)|,
\]
where $$\zeta(x,y)=\sum_{\substack{i:\exists j<i \\f_{j}(x)-f_{j}(y)\neq0}}2^i \left(|f_{i}(x)-f_{i}(y)|-\lfloor| f_{i}(x)-f_{i}(y)|\rfloor\right)\,.$$
\end{lemma}

\begin{proof}
Let $k=2+\lceil \log 1/\eps \rceil$,
and fix some $x_0 \in X$.
For $i \in [k]$, define $F_i : X \to \mathbb R$ by,
\begin{equation}\label{eq:Fi}
F_i(x)=\sum_{j\in \mathbb Z} 2^{jk+i} (f_{jk+i}(x)-f_{jk+i}(x_0))\,.
\end{equation}
It is easy to see that \eqref{eq:converge} implies absolute convergence
of the preceding sum.
We will consider the map $F = F_1 \oplus F_2 \oplus \cdots \oplus F_{k} : X \to \ell_1^k$.
It is straightforward to verify that for every $x,y \in X$,
$$
\|F(x)-F(y)\|_1\leq \sum_{i\in \mathbb Z} 2^i |f_i(x)-f_i(y)|.
 $$

Now, for $i \in [k]$, define
 $$\zeta_i(x,y)=\sum_{\substack{j:\exists \ell<j \\f_{\ell k+i}(x)-f_{\ell k+i}(y)\neq0}}2^{jk+i}( \left|f_{jk+i}(x)-f_{jk+i}(y)|-\lfloor |f_{jk+i}(x)-f_{jk+i}(y)|\rfloor\right)\,.$$
One can easily check that $\sum_{i=1}^k \zeta_i(x,y)\leq \zeta(x,y)$,
thus showing the following for $i \in [k]$ will complete our proof of the lemma,
\begin{equation}\label{eq:fi}
|F_i(x)-F_i(y)| \geq (1-\eps)\sum_{j\in \mathbb Z} \left(2^{jk+i} |f_{jk+i}(x)-f_{jk+i}(y)|\right)-2\zeta_i(x,y).
\end{equation}

Toward this end, fix $i \in [k]$ and $x,y \in X$.
Let $S = \{ j \in \mathbb  Z : |f_{jk+i}(x)-f_{jk+i}(y)| = 1\}$,
and $T = \{ j \in \mathbb Z : 0 < |f_{jk+i}(x)-f_{jk+i}(y)| < 1\}$.
Clearly we then have,
\begin{align*}
\label{eq:Fiequal}
|F_i(x)-F_i(y)|=\left|\sum_{j\in S}2^{jk+i}(f_{jk+i}(x)-f_{jk+i}(y)) +\sum_{j\in T}2^{jk+i}(f_{jk+i}(x)-f_{jk+i}(y))\right|\,.
\end{align*}
If $S \cup T=\emptyset$, then \eqref{eq:fi} is immediate.
Now, suppose that $S \neq \emptyset$, and let
$c = i + k \cdot \max(S)$.  Observe that $\max(S)$ exists by \eqref{eq:converge}.

We then have,
\begin{align*}
\sum_{j\in \mathbb Z}  2^{jk+i} |f_{jk+i}(x)-f_{jk+i}(y)|
&\leq 2^{c}+\mathop{\sum_{j\in S \cup T}}_{j<\max S}  2^{kj+i}+\mathop{\sum_{j\in T}}_{j > \max S}  2^{kj+i}|f_{kj+i}(x)-f_{kj+i}(y)|\\
&\leq 2^{c}+\sum_{j<\max S}  2^{kj+i}+\zeta_i(x,y)\\
&\leq 2^{c}+2\cdot 2^{k(\max S-1)+i}+\zeta_i(x,y)\\
&\leq 2^{c}(1+2^{1-k})+\zeta_i(x,y)\\
&\leq (1+\eps/2)2^{c}+\zeta_i(x,y).
\end{align*}
On the other hand,
\begin{align*}
|F_{i}(x)-F_{i}(y)|&=
\left|\sum_{j\in \mathbb Z} 2^{kj+i}( f_{jk+i}(x)-f_{jk+i}(y))\right|\\
&\geq 2^{c}-\mathop{\sum_{j\in S\cup T}}_{j<\max S}  2^{kj+i}-\mathop{\sum_{j\in T}}_{j > \max S}  2^{kj+i}|f_{kj+i}(x)-f_{kj+i}(y)|\\
&\geq 2^{c}-\sum_{j<\max S}  2^{kj+i}-\zeta_i(x,y)\\
&\geq 2^{c}-2\cdot 2^{k(\max S-1)+i}-\zeta_i(x,y)\\
&\geq 2^{c}(1-2^{1-k})-\zeta_i(x,y)\\
&\geq (1-\eps/2)2^{c}-\zeta_i(x,y).
\end{align*}
Therefore,
\begin{align*}
(1-\eps)\sum_{j\in \mathbb Z}  2^{kj+i}|f_{jk+i}(x)-f_{jk+i}(y)| &\leq
(1-\eps)((1+\eps/2)2^c+\zeta_i(x,y))\\
&\leq {(1-\eps/2)2^c+\zeta_i(x,y)}\\
&\leq |F_{i}(x)-F_{i}(y)|+2\zeta_i(x,y),
\end{align*}
completing the verification of \eqref{eq:fi} in the case when $S \neq \emptyset$.

In the remaining case when $S=\emptyset$ and $T \neq \emptyset$,
if the set $T$ does not have a minimum element, then
$$
\sum_{j\in T} 2^{kj+i}|f_{kj+i}(x)-f_{kj+i}(y)| = \zeta_i(x,y),
$$
making \eqref{eq:fi} vacuous since the right-hand side is non-positive.

Otherwise,
let $\ell=\min(T)$,
and write
\begin{eqnarray*}
|F_i(x)-F_i(y)|&=& \left|\sum_{j\in T} 2^{kj+i} (f_{kj+i}(x)-f_{kj+i}(y))\right |\\
 &\geq& 2^{\ell k+i}|f_{\ell k+i}(x)-f_{\ell k+i}(y)|-\left|\sum_{j\in T,j>\ell} 2^{kj+i} (f_{kj+i}(x)-f_{kj+i}(y))\right |\\
&\geq&2^{\ell k+i} |f_{\ell k+i}(x)-f_{\ell k+i}(y)|-\zeta_i(x,y)\\
&=&\sum_{j\in \mathbb Z} 2^{kj+i} |f_{kj+i}(x)-f_{kj+i}(y)|-2\,\zeta_i(x,y)\,.
\end{eqnarray*}
This completes the proof.
\end{proof}


In Section \ref{sec:embedding}, we will require the following straightforward corollary.

\begin{corollary}\label{col:scales}
For every $\e \in (0,1)$ and $m \in \mathbb N$, the following holds.
Let $(X,d)$ be a metric space, and suppose we have a
family of functions $\{f_i:X\to [0,1]^m\}_{i \in \mathbb Z}$ such that for all $x,y\in X$,
$$\sum_{i \in \mathbb Z} 2^i \|f_i(x)-f_i(y)\|_1 < \infty\,.$$
Then there exists a mapping $F : V \to \ell_1^{m(2+\lceil \log \frac{1}{\e} \rceil)}$ such that
for all  $x,y\in X$,
\[
(1-\eps)\sum_{i \in \mathbb Z} \left(2^{i} \|f_i(x)-f_i(y)\|_1\right)- 2\,\zeta(x,y)  \leq \|F(x)-F(y)\|_1\leq \sum_{i \in \mathbb Z} 2^{i}\|f_i(x)-f_i(y)\|_1,
\]
where
\begin{equation}\label{eq:zetadef}
\zeta(x,y)=\sum_{k=1}^m\sum_{\substack{i:\exists j<i \\f_{j}(x)_k-f_{j}(y)_k\neq0}}2^i(|f_{i}(x)_k-f_{i}(y)_k|-\lfloor |f_{i}(x)_k-f_{i}(y)_k|\rfloor),
\end{equation}
and we have used the notation $x_k$ for the $k$th coordinate of $x \in \mathbb R^m$.
\end{corollary}

\section{Scale assignment}\label{sec:assignments}
\label{sec:assignment}

Let $T=(V,E)$ be a metric tree with root $r \in V$, equipped with a monotone coloring $\chi : E \to \mathbb N$.
We will now describe a way of assigning ``scales'' to the vertices of $T$.
These scale values will be used in Section \ref{sec:embedding} to guide our
eventual embedding.
The scales of a vertex will describe, roughly, the subset and magnitude
of coordinates that should differ between the vertex and its neighbors.
First, we fix some notation.

For every $c \in \chi(E)$, we use $\gamma_c$ to denote the path in $T$ colored $c$,
and we use $v_c$ to denote the vertex of $\gamma_c$ which is closest to the root.
We will also use the notation $T(c)$ to denote the
subtree of $T$ under the color $c$; formally, $T(c)$ is the induced (rooted)
subtree on $\{v_c\} \cup V(T_u)$ where $u \in V$ is the child of $v_c$
such that $\chi(v_c,u)=c$, and $T_u$ is the subtree rooted at $u$.

We will write $p(v)$ for the parent of a vertex $v \in V$, and $p(r)=r$.
Furthermore, we define the ``parent color'' of a color class by
$\rho(c) = \chi(v_c, p(v_c))$ with the convention that $\chi(r,r)=c_0$,
where $c_0 \in \mathbb N \setminus \chi(E)$ is some fixed element.
Finally, we put $T(c_0)=T$.

\subsection{Scale selectors}

We start by defining a function $\kappa:\chi(E) \cup \{c_0\} \to \mathbb N$ which describes the ``branching factor'' for each color class,
\begin{equation}
\kappa(c)=\left\lfloor\log_2 {|E(T(\rho(c)))|\over |E(T(c))|}\right\rfloor+1.
\end{equation}
Moreover, we define $\varphi : \chi(E) \cup \{c_0\} \to \mathbb N \cup \{0\}$ inductively
by setting $\varphi(c_0)=0$, and
\begin{equation}\label{def:phi}
\varphi(c)=\kappa(c)+\varphi(\rho(c)),
\end{equation}
for $c\in \chi(E)$.

Observe that for every color $c \in \chi(E)$, we have,
\begin{align}\label{eq:kappa}
\varphi(c)= \sum_{c'\in \chi(E(P_{v_c}))\cup \{c\}}\hspace{-5mm} \kappa(c') \leq \sum_{c'\in \chi(E(P_{v_c}))\cup \{c\}} \left(1+\log_2 {|E(T(\rho(c')))|\over |E(T(c'))|}\right)\leq M(\chi)+\log_2 |E|.
\end{align}

Next, we use $\varphi$ to inductively define our scale selectors.
Let $$m(T)=\min\{ \len(e): e\in E \textrm{ and } \len(e)>0\}.$$
We now define a family of functions $\{\tau_i : V \to \mathbb N \cup \{0\}\}_{i \in \mathbb Z}$.

For $v\in V$, let $c=\chi(v,p(v))$, and
put $\tau_i(v)=0$ for $i<\left \lfloor \log_2 \left(m(T)\over M(\chi)+\log_2 |E|\right )\right\rfloor$, and otherwise,
\begin{align}
\tau_i(v)=
\min\Bigg(&\underbrace{\left\lceil d_T(v,v_c)-\min\left(d_T(v,v_c), \sum_{j=-\infty}^{i-1} 2^{j}\tau_j(v)\right)\over 2^i\right\rceil}_{(A)}, \underbrace{\varphi(c)-\sum_{c'\in\col(E(P_{v}))} \tau_i(v_{c'}) }_{(B)}\label{def:tau}\Bigg )\,.
\end{align}

The value of $\tau_i(v)$ will be used in Section~\ref{sec:embedding} to determine how many coordinates
 of magnitude $\asymp 2^i$ change as the embedding proceeds from $v_c$ to $v$.
In this definition, we try to cover the distance from root to $v$ with the smallest scales possible while satisfying the inequality
$$\varphi(c)\geq \tau_i(v)+\sum_{c'\in\chi(E(P_{v}))} \tau_i(v_{c'}).
$$

For $v \in V\setminus \{r\}$, let $c=\col(v,p(v))$, for each $i \in \mathbb Z$, part (B) of \eqref{def:tau} for $\tau_i(v_c)$ implies that
$$\tau_i(v_c)\leq \varphi(\rho(c))-\sum_{c'\in\col(E(P_{v_c}))} \tau_i(v_{c'}).$$
Hence,
\begin{eqnarray}
\varphi(c)-\sum_{c'\in\col(E(P_{v}))} \tau_i(v_{c'})
&=& \varphi(c)-\tau_i(v_c)-\sum_{c'\in\chi(E(P_{v_c}))} \tau_i(v_{c'})\nonumber\\
&\geq&\varphi(c)-\varphi(\rho(c))\nonumber\\ &=& \kappa(c)\nonumber\\
&\geq& 1\label{eq:tau:b}.
\end{eqnarray}
Therefore, part $(B)$ of \eqref{def:tau} is always positive, so if $\tau_k(v)=0$ for some $k\geq \left \lfloor \log_2 \left(m(T)\over M(\chi)+\log_2 |E|\right )\right\rfloor$, then $\tau_k(v)$ is defined by part $(A)$ of (18). Hence $\sum_{j=-\infty}^{i-1} 2^{j}\tau_j(v) \geq d_T(v,v_c)$ and the following observation is immediate.

\begin{observation}\label{obv:tau:zero}
For $v\in V$ and $k\geq \left \lfloor \log_2 \left(m(T)\over M(\chi)+\log_2 |E|\right )\right\rfloor$, if $\tau_k(v)=0$ then for all $i\geq k$, $\tau_i(v)=0$.
\end{observation}

Comparing part $(A)$ of \eqref{def:tau} for $\tau_i(v)$ and $\tau_{i+1}(v)$ also allows us to observe the following.

\begin{observation}\label{obv:tau:zero2}
For $v\in V$ and $k\geq \left \lfloor \log_2 \left(m(T)\over M(\chi)+\log_2 |E|\right )\right\rfloor$, if part (A) in \eqref{def:tau} for $\tau_k(v)$ is less than or equal to part (B) then for all $i> k$, $\tau_i(v)=0$.
\end{observation}

\subsection{Properties of the scale selector maps}

We now prove some key properties of the maps $\kappa, \varphi$, and $\{\tau_i\}$.

\begin{lemma}\label{obv:tau}
For every vertex $v\in V$ with $c=\col(v,p(v))$, the following holds.  For all ${i}\in \mathbb Z$ with ${{d_T(v,v_c)}\over \kappa(c)} \leq 2^{i-1}$, we have
$\tau_i(v)=0.$

\end{lemma}

\begin{proof}
If $d_T(v,v_c)=0$, the lemma is vacuous.
Suppose now that $d_T(v,v_c)>0$, and let $k=\left\lceil \log_2\left({{d_T(v,v_c)}\over \kappa(c)}\right)\right\rceil.$ We have $d_T(v,v_c)\geq m(T)$ and $\kappa(c)\leq \log_2 |E|+1$, therefore
$$k\geq \left \lfloor \log_2 \left(m(T)\over M(\chi)+\log_2 |E|\right )\right\rfloor.$$
It follows that for $i \geq k$, $\tau_i(v)$ is given by \eqref{def:tau}.

If $\tau_k(v)=0$, then by Observation~\ref{obv:tau:zero},  for all $i\geq k$, $\tau_i(v)=0$.

On the other hand if $\tau_k(v)\neq 0$ then either it is determined by part (B) of \eqref{def:tau}, in which case
 $$\tau_k(v)=\varphi(c)-\sum_{c'\in\col(E(P_{v}))} \tau_k(v_{c'})
= \varphi(c)-\tau_k(v_c)-\sum_{c'\in\chi(E(P_{v_c}))} \tau_k(v_{c'})
\geq\varphi(c)-\varphi(\rho(c)) = \kappa(c),
$$
 implying that
$$
\sum_{j=-\infty}^{k} 2^j\tau_j(v)\geq \kappa(c)2^k\geq d_T(v,v_c)\,.
$$
Examining part (A) of \eqref{def:tau}, we see that $\tau_{k+1}(v)=0$, and by Observation~\ref{obv:tau:zero}, $\tau_i(v)=0$ for $i > k$.
Alternately, $\tau_k(v)$ is determined by part (A) of \eqref{def:tau}, and by Observation~\ref{obv:tau:zero2} $\tau_{i}(v)=0$ for $i>k$, completing the proof.

\end{proof}

The next lemma shows how the values $\{\tau_i(v)\}$ track the distance from $v_c$ to $v$.

\begin{lemma}\label{obv:tau2}
For any vertex $v\in V$ with $c=\col(v,p(v))$, we have
$$d_T(v,v_c)\leq \sum_{i=-\infty}^\infty 2^i\tau_i(v) \leq 3\,d_T(v,v_c).$$
\end{lemma}

\begin{proof}
If $d_T(v,v_c)=0$, the lemma is vacuous. Suppose now that $d_T(v,v_c)>0$, and let $$k=\max \{i:\tau_i(v)\neq 0\}.$$ By Lemma~\ref{obv:tau}, the maximum exists.

We have $\tau_{k+1}(v)=0$, and thus inequality \eqref{eq:tau:b} implies that
part (A) of \eqref{def:tau} specifies $\tau_{k+1}(v)$, yielding
$$d_T(v,v_c)\leq \sum_{i=-\infty}^{k} 2^i\tau_i(v) = \sum_{i=-\infty}^{\infty} 2^i\tau_i(v).$$

On the other hand, since $\tau_k(v)>0$, we must have $d_T(v,v_c)>\sum_{i=-\infty}^{k-1} 2^i\tau_i(v),$ and Lemma~\ref{obv:tau} implies that $2^k< 2\,d_T(v,v_c),$ hence,
\begin{align*}
\sum_{i=-\infty}^k 2^i\tau_i(v)
&\leq \sum_{i=-\infty}^{k-1} 2^i\tau_i(v)+2^k\left \lceil d_T(v,v_c)-\sum_{i=-\infty}^{k-1} 2^i\tau_i(v)\over 2^k\right\rceil\\
&< \sum_{i=-\infty}^{k-1} 2^i\tau_i(v) +2^k\left ( {d_T(v,v_c)-\sum_{i=-\infty}^{k-1} 2^i\tau_i(v)\over 2^k}+1\right)\\
&= \sum_{i=-\infty}^{k-1} 2^i\tau_i(v) +2^k+\left ( d_T(v,v_c)-\sum_{i=-\infty}^{k-1} 2^i\tau_i(v)\right)\\
&\leq d_T(v,v_c)+2^k\\
&< 3\,d_T(v,v_c).
\end{align*}
\end{proof}
The following lemma shows that for any color $c\in \chi(E)$ the value of $\tau_i$ does not decrease as we move further from $v_c$ in $\gamma_c$.
\begin{lemma}\label{lem:incd}
Let $u,w\in V$ be such that $c=\col(w,p(w))=\col(u,p(u))$, and $d_T(w,v_c)\leq d_T(u,v_c)$.   Then for all $i\in \mathbb Z$, we have
$$\tau_i(w)\leq \tau_i(u).$$
\end{lemma}

\begin{proof}

First let $k$ be the smallest integer for which,
$$\left\lceil d_T(w,v_c)-\min\left(d_T(w,v_c), \sum_{j=-\infty}^{k-1} 2^{j}\tau_j(w)\right)\over 2^k\right\rceil\leq \varphi(c)-\sum_{c'\in\col(E(P_{w}))} \tau_k(v_{c'}).$$

This $k$ exists since, by \eqref{eq:tau:b}, the right hand side is always positive, while by Lemma~\ref{obv:tau}, the left hand side must be zero for some $k \in \mathbb{Z}$.

For $i>k$, by Observation~\ref{obv:tau:zero2} we have, $\tau_i(w)=0$. Therefore, for $i>k$, we have $\tau_i(u)\geq \tau_i(w)$. We now use induction on $i$ to show that for $i< k$,  $\tau_i(u)=\tau_i(w)$, and for  $i=k$, $\tau_k(u)\geq \tau_k(w)$. Recall that, for $i<\left \lfloor \log_2 \left(m(T)\over M(\chi)+\log_2 |E|\right )\right\rfloor$, we have $\tau_i(w)=\tau_i(u)=0$, which gives us the base case of the induction.

Now, by definition of $k$,
part (B) of \eqref{def:tau} for $\tau_{k-1}(w)$ is an integer strictly less than part (A), hence
\begin{align}
 \sum_{j=-\infty}^{k-1} 2^{j}\tau_j(w)
&= 2^{k-1}\tau_{k-1}(w) + \sum_{j=-\infty}^{k-2} 2^{j}\tau_j(w) \nonumber \\
&\leq 2^{k-1}\left(\left\lceil d_T(w,v_c)- \sum_{j=-\infty}^{k-2} 2^{j}\tau_j(w)\over 2^{k-1}\right\rceil-1\right) + \sum_{j=-\infty}^{k-2} 2^{j}\tau_j(w) \nonumber \\
&<  2^{k-1}\left( d_T(w,v_c)- \sum_{j=-\infty}^{k-2} 2^{j}\tau_j(w)\over 2^{k-1}\right) + \sum_{j=-\infty}^{k-2} 2^{j}\tau_j(w) \nonumber \\
&\leq d_T(w,v_c)\,. \label{eq:incd:tmp}
\end{align}
For $ \left \lfloor \log_2 \left(m(T)\over M(\chi)+\log_2 |E|\right )\right\rfloor\leq i\leq k$, by \eqref{eq:incd:tmp}, and as $d_T(u,v_c) \geq d_T(w,v_c)$, we have

\begin{align}
&\min\bigg(d_T(w,v_c), \sum_{j=-\infty}^{i-1} 2^{j}\tau_j(w)\bigg)= \sum_{j=-\infty}^{i-1} 2^{j}\tau_j(w)= \min\bigg(d_T(u,v_c), \sum_{j=-\infty}^{i-1} 2^{j}\tau_j(w)\bigg).\label{eq:uw}
\end{align}
By our induction hypothesis for all $j<i$, $\tau_j(w)=\tau_j(u)$, so using $\eqref{eq:uw}$ we can write,
\begin{align}
& d_T(w,v_c)-\min\bigg(d_T(w,v_c), \sum_{j=-\infty}^{i-1} 2^{j}\tau_j(w)\bigg)\leq  d_T(u,v_c)-\min\bigg(d_T(u,v_c), \sum_{j=-\infty}^{i-1} 2^{j}\tau_j(u)\bigg).\label{eq:dist:tau}
\end{align}
 Since $\col(w,p(w))=\col(u,p(u))$, for all $i\in \mathbb Z$ part (B) of \eqref{def:tau} is identical for $\tau_i(u)$ and $\tau_i(w)$. Therefore, using \eqref{eq:dist:tau}, and the definition of $k$, for all $\left \lfloor \log_2 \left(m(T)\over M(\chi)+\log_2 |E|\right )\right\rfloor \leq i<k$, part (B) of \eqref{def:tau} specifies $\tau_i(u)$ and $\tau_i(w)$, hence
$$ \tau_i(u)=\tau_i(w)=\varphi(c)-\sum_{c'\in\col(E(P_{w}))} \tau_i(v_{c'}).$$

For the case that $i=k$, part (B) of  \eqref{def:tau} is identical for $\tau_k(u)$ and $\tau_k(w)$, and inequality \eqref{eq:dist:tau} implies that part (A) of  \eqref{def:tau} for $\tau_k(u)$ is at least as large as part (A) of \eqref{def:tau} for $\tau_k(w)$, completing the proof.
\end{proof}

The next lemma bounds the distance between two vertices in the graph based on $\{\tau_i\}$.
\begin{lemma}\label{lem:tau3}
Let $k> \left \lfloor \log_2 \left(m(T)\over M(\chi)+\log_2 |E|\right )\right\rfloor$ be an integer. For any two vertices  $w$  and $u$ such that $\tau_k(u)\neq 0$, $\tau_{k-1}(w)=0$ and $\col(w,p(w))=\col(u,p(u))$, we have
$$d_T(u,w)> 2^{k-1}.$$
\end{lemma}

\begin{proof}
By Observation~\ref{obv:tau:zero}, $\tau_k(w)=0$.  Letting $c=\col(u,p(u))$, by Lemma~\ref{lem:incd} we have $d_T(v_c,u)\geq  d_T(v_c,w)$. Using Lemma~\ref{lem:incd} again, we can conclude that  for all $i\in \mathbb Z$, $\tau_i(u)\geq \tau_i(w)$.
Since $\tau_{k-1}(w)=0$, inequality \eqref{eq:tau:b} implies that part (A) of  \eqref{def:tau} specifies $\tau_{k-1}(w)$. Therefore,
\begin{align}
d_T(w,v_c)&\leq \sum_{i=-\infty}^{k-2} 2^{i}\tau_i(w) \nonumber \\
&\leq \sum_{i=-\infty}^{k-2} 2^{i}\tau_i(u) \nonumber \\
&= \left(\sum_{i=-\infty}^{k-1} 2^{i}\tau_i(u)\right)- 2^{k-1}\tau_{k-1}(u). \label{eq:461}
\end{align}
Since $\tau_k(u)>0 $, using part $(A)$ of \eqref{def:tau}, we can write
\begin{equation}
\label{eq:462}
d_T(u,v_c)> \sum_{i=-\infty}^{k-1} 2^{i}\tau_i(u).
\end{equation}
Observation~\ref{obv:tau:zero} implies that $\tau_{k-1}(u)\neq 0$, thus $\tau_{k-1}(u) \geq 1$, and using \eqref{eq:461} and \eqref{eq:462}, we have
\[
d_T(w,u)= d_T(u,v_c)-d_T(w,v_c)>2^{k-1 },
\]
completing the proof.
\end{proof}
The next lemma and the following two corollaries bound the number of colors $c$ in the tree which have a small value of $\varphi(c)$.
\begin{lemma}\label{lem:count}
For any  $k\in \mathbb N\cup \{0\}$, and any color $c\in \mathcal \chi(E)$, we have
\[
\#\{c'\in\chi(E(T(c))): \varphi(c')-\varphi(c)=k\}  \leq  2^{k}\,.
\]
\end{lemma}

\begin{proof}
We start the proof by comparing the size of the subtrees $T(c')$ and $T(c)$ for $c'\in \col(E(T(c)))$.

For a given color $c'\in \col(E(T(c)))$, we define the sequence $\{c_i\}_{i\in \mathbb N}$ as follows. We put $c_1=c'$ and for   $i>1$ we put $c_i=\rho(c_{i-1})$. Suppose now that $c_m=c$, we have
\begin{align}\label{eq:colorbound}
\varphi(c_m)-\varphi(c_1)&=\sum_{i=1}^{m-1} \kappa(c_i)\nonumber\\
&\geq \sum_{i=1}^{m-1}\log_2\left( {|E(T({c_{i+1}))}|\over |E(T({c_{i}))}|}\right)\nonumber\\
&\geq \log_2\left( {|E(T(c))|\over |E(T({c'}))|}\right).
\end{align}
This inequality implies that
$$|E(T(c))|\leq 2^{\varphi(c')-\varphi(c)}|E(T(c'))|.$$
It is easy to check that for colors $a,b\in \chi(E(T(c)))$ such that $\varphi(a)=\varphi(b)$, subtrees $T(a)$ and $T(b)$ are edge disjoint. Therefore, for $k \in \mathbb N\cup \{0\}$, summing over all the colors $c'$ such that $\varphi(c')-\varphi(c)=k$ gives

\[ \#\{c'\in\chi(E(T(c))): \varphi(c')-\varphi(c) = k\}\, \leq \!\!\sum_{\substack{c' \in \chi(E(T(c)))\\ \varphi(c')-\varphi(c)=k}} \,{2^k\,|E(T(c'))|\over |E(T(c))|}=2^k\!\!\!\!\sum_{\substack{c' \in \chi(E(T(c)))\\ \varphi(c')-\varphi(c)=k}} \,{|E(T(c'))|\over |E(T(c))|} \leq 2^k \,. \]
\end{proof}
The following two corollaries are immediate from Lemma~\ref{lem:count}.
\begin{corollary}\label{cor:count}
For any  $k\in \mathbb N$, and any color $c\in \mathcal \chi(E)$, we have
\[
\#\{c'\in\chi(E(T(c))): \varphi(c')-\varphi(c)\leq k\} \quad < \quad 2^{k+1}.
\]
\end{corollary}

\begin{corollary}\label{cor:expsum}
For any color $c\in \mathcal \chi(E)$, and constant $C\geq 2$, we have
\[
\sum_{c'\in \chi(E(T(c)))\setminus \{c\}} 2^{-C(\varphi(c')-\varphi(c))}< 2^{2-C}.
\]
\end{corollary}

The next lemma is similer to Lemma~\ref{lem:tau3}.
The assumption is more general, and the conclusion is correspondingly weaker.
This result is used primarily to enable the proof of Lemma~\ref{lem:triangle}.

\begin{lemma}\label{lem:div}
Let $u\in V$ and $w\in V(P_u)$ be such that $\varphi(\chi(u,p(u)))>\varphi(\chi(w,p(w)))$. For all vertices $x\in V(T_u)$, and $k\in \mathbb Z$ with
\begin{equation}\label{eq:lem:div}
2^{k}>\left({6\,d_T(x,w)\over \varphi(\col(u,p(u)))-\varphi(\col(w,p(w)))}\right),
\end{equation}
we have
$
\tau_k(x)=0.
$
\end{lemma}

\begin{proof}
In the case that $d_T(x,w)=0$,  this lemma is vacuous. Suppose now that $d_T(x,w)>0$.
Let $c_1,\ldots, c_m$ be the set of colors that appear on the path $P_{x\,p(w)}$, in order from $x$ to $p(w)$, and for $i\in[m]$, let $y_i=v_{c_i}$. We prove this lemma by showing that if,
\begin{equation}\label{eq:tmp:div}
k\geq\log_2 \left({6\,d_T(x,w)\over \varphi(\col(u,p(u)))-\varphi(\col(w,p(w)))}\right),
\end{equation}
then part $(A)$ of \eqref{def:tau}  for $\tau_{k}(x)$ is zero.

First note that, $\varphi(\col(u,p(u)))-\varphi(\col(w,p(w)))\leq M(\chi)+\log_2 |E|$ and $d_T(x,w)\geq m(T)$, hence \eqref{eq:tmp:div} implies
$$
k \geq\left \lfloor \log_2 \left(m(T)\over M(\chi)+\log_2 |E|\right )\right\rfloor.
$$
By Lemma~\ref{obv:tau2}, we have
\begin{equation}\label{eq:obv:tau2}
\sum_{i=1}^{m-2}  2^{k-1}\tau_{k-1}(y_{i})\leq \sum_{i=1}^{m-2} \sum_{j=-\infty}^{\infty} 2^{j}\tau_j(y_{i})\leq \sum_{i=1}^{m-2} 3\,d_T(y_{i},y_{{i+1}})= 3\,d_T(y_{1},y_{{m-1}}).
\end{equation}
Now, using \eqref{eq:lem:div} gives
\begin{align}
\varphi(c_1)-\varphi(c_{m})&\geq  \varphi(\col(u,p(u)))-\varphi(\col(w,p(w)))\nonumber\\
&\geq {6\,d_T(x,w)\over 2^{k}}\nonumber\\
&\geq {6\,d_T(x,y_{{m-1}})\over 2^{k}}.\label{lem:phi:tmp}
\end{align}
Using the above inequality and  \eqref{eq:obv:tau2}, we can write
\begin{align*}
d_T(x,y_{1})&= d_T(x,y_{m-1})-d_T(y_{1},y_{{m-1}})\\
&\leq{2^{k-1}\over 3}\left( \varphi(c_1)-\varphi(c_{m})-\sum_{i=1}^{m-2}  \tau_{k-1}(y_{i})\right).
\end{align*}
First, note that $c_m=\chi(y_{m-1},p(y_{m-1}))$. Now, we use part $(B)$ of \eqref{def:tau} for $\tau_k(y_{m-1})$ to write
\begin{align}
d_T(x,y_{1})&\leq{2^{k-1}\over 3}\left( \varphi(c_1)-\left(\tau_{k-1}(y_{m-1})+\sum_{c'\in \chi(E(P_{y_{{m-1}}}))}  \tau_{k-1}(v_{c'})\right)-\sum_{i=1}^{m-2}  \tau_{k-1}(y_{i})\right)\nonumber\\
&\leq{2^{k-1}\over 3}\left( \varphi(c_1)-\sum_{c'\in \col(E(P_{x}))}  \tau_{k-1}(v_{c'})\right)\nonumber\\
&\leq{2^{k-1}}\left( \varphi(\col({x},p(x)))-\sum_{c'\in \col(E(P_{x}))}  \tau_{k-1}(v_{c'})\right).\label{eq:lem:div:last}
\end{align}
Therefore, either part (A) of \eqref{def:tau} specifies  $\tau_{k-1}(x)$ in which case by Observation~\ref{obv:tau:zero2},
$\tau_{i}(v)=0$ for $i\geq k$, or part (B) of \eqref{def:tau} specifies $\tau_{k-1}(x)$ in which case by \eqref{eq:lem:div:last} we have,
$$\tau_{k-1}(x)2^{k-1}\geq  d_T(x,y_{1}),$$
and part (A) of \eqref{def:tau} is zero for $i\geq k$.
\end{proof}
In Section~\ref{sec:embedding}, we give the description of our embedding and analyze its distortion. In the analysis of embedding, for a given pair of vertices $x,y\in V$, we divide the path between $x$ and $y$ into subpaths and for each subpath we show that either the contribution of that subpath to the distance between $x$ and $y$ in the embedding is ``large'' through a concentration of measure argument, or we use the following lemma to show that the length of the subpath is ``small,'' compared to the distance between $x$ and $y$. The complete argument is somewhat more delicate and one can find the details of how Lemma~\ref{lem:triangle} is used in the proof of
Lemma~\ref{lem:sec5:main:induction}.

\begin{lemma}\label{lem:triangle}
There exists a constant $C > 0$ such that the following holds.
For any $c \in \mathbb \chi(E)$ and $v\in V(T({c}))$, and for any $\varepsilon\in(0,\frac 1 2]$,
there are vertices $u,u' \in V$ with $u \neq u'$ and $d_T(u,v) \leq \eps\,d_T(u,u')$, and such that,
\begin{eqnarray*}
u,u'&\in & \{v_a:a\in \chi(E(P_{v \,v_c}))\}\cup\{v\} .
\end{eqnarray*}
Furthermore,
for all vertices $x\in V(P_{u'u})\setminus \{u'\}$, for all $k \in \mathbb Z$,
\[
\tau_k(x) \neq 0 \implies
2^k < \left({C{d_T(u,u')} \over \eps(\varphi(\chi(u,p(u)))-\varphi(\chi(v_c,p(v_c))))}\right).
\]
\end{lemma}

\begin{proof}
Let $r'=v_c$, and let $c_1,\ldots, c_m$ be the set of colors that appear on the path $P_{vr'}$ in order from $v$ to $r'$, and put $c_{m+1}=\col(r',p(r'))$. We define $y_0=v$, and for $i\in[m]$, $y_i=v_{c_i}$. Note that
$\{y_0,\ldots, y_m\}=\{v\}\cup\{v_a:a\in \chi(E(P_{v \,v_c}))\}$, and for $i\leq m$, $\chi(y_i,p(y_i))=c_{i+1}$.
We give a constructive proof for the lemma.

For $i\in \mathbb N$, we construct a sequence $(a_i,b_i)\in{\mathbb N}\times \mathbb N$, the idea being that $P_{y_{a_i},y_{b_i}}$ is a nonempty subpath $P_{vr'}$ such that for different values of $i$, these subpaths are edge disjoint. At each step of construction either we can use $(a_i,b_i)$ to find $u$ and $u'$ such that they satisfy the properties of this lemma, or we find $(a_{i+1},b_{i+1})$ such that $b_{i+1}<b_i$. The last condition guarantees that we can always find $u$ and $u'$ that satisfy conditions of this lemma.

We start with $a_1=m$ and $b_1=m-1$. If
$d_T(v,y_{b_1})\leq \eps d_T(y_{a_1},y_{b_1})$ then
\[
 \left({2{d_T(y_m,y_{m-1})} \over \varphi(\chi(y_{m-1},p(y_{m-1})))-\varphi(\chi(r',p(r')))}\right) = {{2d_T(y_{a_1},y_{b_1})}\over \kappa(c)}
\]
and  by Lemma~\ref{obv:tau}  the assignment $u'=y_{a_1}$ and $u=y_{b_1}$ satisfies the conditions of this lemma if $C \geq \frac{1}{2}$.
Otherwise, for $i \geq 1$, we choose $(a_{i+1},b_{i+1})$ based on $(a_i,b_i)$, and construct the rest of the sequence preserving the following three properties:
\begin{enumerate}
\item  $\varphi(c_{b_i+1})-\varphi(c_{a_i+1})\geq \varphi(c_{a_i+1} )-\varphi(\chi(r',p(r')))$;
\item  $d_T(y_{b_i},v) \geq \eps d_T(y_{b_i},y_{a_i})$;
\item $a_i>b_i$.
\end{enumerate}
Let ${j}\in \{0,\ldots, m\}$ be the maximum integer such that $\eps d_T(y_{j},y_{b_i})\geq d_T(v,y_{j})$. Note that $j<b_i$, and the maximum always exists because $y_0 = v$.
We will now split the proof into three cases.

\medskip
\noindent
{\bf Case I:  $\varphi(c_{j+2})-\varphi(c_{b_i+1})\geq 2(\varphi(c_{b_i+1})-\varphi(c_{a_{i}+1})).$}

\medskip

In this case by condition (iii), $\varphi(c_{b_i+1})-\varphi(c_{a_{i}+1})>0$. Hence $j+1<b_i$, and we can preserve conditions (i), (ii) and (iii) with $$(a_{i+1},b_{i+1})=(b_i, {j+1}).$$

\medskip
\noindent
{\bf Case II: $\varphi(c_{j+2})-\varphi(c_{b_i+1})< 2(\varphi(c_{b_i+1})-\varphi(c_{a_{i}+1}))$ and
$\varphi(c_{j+1})-\varphi(c_{b_i+1})\geq 6(\varphi(c_{b_i+1})-\varphi(c_{a_i+1}))$.}

\medskip

In this case by \eqref{def:phi} we have,
\[\kappa(c_{j+1})= \varphi(c_{j+1})-\varphi(c_{{j+2}})=(\varphi(c_{j+1})-\varphi(c_{b_i+1}))-(\varphi(c_{j+2})-\varphi(c_{b_i+1})).\]
Using the conditions of this case, we write
\begin{align*}
\kappa(c_{j+1})
&= (\varphi(c_{j+1})-\varphi(c_{b_i+1}))-(\varphi(c_{j+2})-\varphi(c_{b_i+1}))\\
&\geq 6(\varphi(c_{b_i+1})-\varphi(c_{a_{i}+1}))-(\varphi(c_{j+2})-\varphi(c_{b_i+1}))\\
&= \Big(2(\varphi(c_{b_i+1})-\varphi(c_{a_{i}+1})) +4(\varphi(c_{b_i+1})-\varphi(c_{a_i+1}))\Big)-\Big(\varphi(c_{j+2})-\varphi(c_{b_i+1})\Big)\\
&> \Big(2(\varphi(c_{b_i+1})-\varphi(c_{a_{i}+1})) +2(\varphi(c_{j+2})-\varphi(c_{b_i+1}))\Big)-\Big(\varphi(c_{j+2})-\varphi(c_{b_i+1})\Big),
 \end{align*}
and by condition (i),
 \begin{align}\label{eq:triangle:kappa}
 \kappa(c_{j+1})
&> \Big(\big(\varphi(c_{b_i+1})-\varphi(c_{a_{i}+1})\big)+\big(\varphi(c_{a_i+1})-\varphi(\chi(r',p(r'))\big)
+2(\varphi(c_{j+2})-\varphi(c_{b_i+1}))\Big)\nonumber
\\&\hspace{2.0cm}-\Big(\varphi(c_{j+2})-\varphi(c_{b_i+1})\Big)\nonumber\\
&= \varphi(c_{j+2})-\varphi(\chi(r',p(r'))).
\end{align}

    Thus if $ d_T(y_{{j+1}},v) \geq \eps\, d_T(y_j,y_{{j+1}})$, then  $(a_{i+1},b_{i+1})=({j+1},j)$, satisfies condition (i) by \eqref{eq:triangle:kappa}, and it is also easy to verify that it satisfies conditions (ii) and (iii).
  If $ d_T(y_{{j+1}},v) < \eps\, d_T(y_{j},y_{{j+1}})$, then by  \eqref{def:phi},
\[  \varphi(\chi(y_{j},p(y_{j})))=\varphi(c_{j+1})=\kappa(c_{j+1})+\varphi(c_{j+2})\]
 and by \eqref{eq:triangle:kappa},
\begin{align*}
 \left({2{d_T(y_j,y_{j+1})} \over (\varphi(\chi(y_{j},p(y_{j})))-\varphi(\chi(r',p(r'))))}\right) &=
  \left({2{d_T(y_j,y_{j+1})} \over  \kappa(c_{j+1})+\varphi(c_{j+2})-\varphi(\chi(r',p(r')))}\right)\\
 &>
 {{d_T(y_{j},y_{{j+1}})}\over  \kappa(c_{j+1})}.
\end{align*}
Hence Lemma~\ref{obv:tau} implies that the assignment $u'=y_{{j+1}}$ and $u=y_{j}$ satisfies the conditions of this lemma if $C \geq \frac{1}{2}$.

\medskip
\noindent
{\bf Case III: $\varphi(c_{j+1})-\varphi(c_{b_i+1})< 6(\varphi(c_{b_i+1})-\varphi(c_{a_i+1}))$.}

\medskip

In this case we use
Lemma~\ref{lem:div} to show that the assignment $u=y_{j}$ and $u'=y_{b_i}$ satisfies the conditions of the lemma.
We have
\begin{align*}
\varphi(\chi(y_j,p(y_j)))-\varphi(\chi(r',p(r')))
&=
\varphi(c_{j+1})-\varphi(\chi(r',p(r')))\\
&= (\varphi(c_{j+1}-\varphi(c_{{b_i}+1}))+(\varphi(c_{b_i+1})-\varphi(c_{a_i+1}))\\&\qquad+(\varphi(c_{a_i+1})-\varphi(\chi(r',p(r'))))\\
&< 6(\varphi(c_{b_i+1})-\varphi(c_{{a_i}+1}))+(\varphi(c_{b_i+1})-\varphi(c_{a_i+1}))\\&\qquad+(\varphi(c_{a_i+1})-\varphi(\chi(r',p(r')))),
\end{align*}
and by condition (i),
\[\varphi(\chi(y_j,p(y_j)))-\varphi(\chi(r',p(r'))) < 8(\varphi(c_{b_i+1})-\varphi(c_{a_i+1})).\]

Condition (ii) and the definition of $y_j$ imply that,
$$d_T(y_{j},y_{b_i})\geq {(1-\eps)d_T(v,y_{b_i})}\geq \eps{(1-\eps)d_T(y_{a_i},y_{b_i})}\geq {\eps\over 2}\,d_T(y_{a_i},y_{b_i}).$$
Hence,
$$ \left({6({2\over \eps}){d_T(y_j,y_{b_{i}})} \over {1\over 8}(\varphi(\chi(y_{j},p(y_{j})))-\varphi(\chi(r',p(r'))))}\right)\geq
\left({6d_T(y_{b_i},y_{a_i})\over \varphi(c_{b_i+1})-\varphi(c_{a_i+1})}\right),$$
and by applying Lemma~\ref{lem:div} with $u=y_{b_i}$ and $w=y_{a_i}$, we can conclude that the assignment $u=y_{j}$ and $u'=y_{b_i}$ satisfies the conditions of this lemma with $C=96$.
\end{proof}

\section{The embedding}\label{sec:embedding}

   We now present a proof of Theorem~\ref{thm:allbut}, thereby completing the proof of Theorem \ref{thm:main}.
   We first introduce a random embedding of the tree $T$ into $\ell_1$, and then show that, for a suitable choice of parameters,
   with non-zero probability our construction satisfies the conditions of the theorem.

\medskip
\noindent
{\bf Notation:}
We use the notations and definitions introduced in Section~\ref{sec:assignments}.
Moreover, in this section, for $c\in \chi(E)\cup \{\col(r,p(r))\}$, we use $\rho^{-1}(c)$ to denote the set of colors $c'\in \chi(E)$ such that $\rho(c')=c$, i.e. the colors of the ``children'' of $c$.
For $m,n\in \mathbb N$, and $A\in \mathbb R^{m\times n}$, we use the notation $A[i]$ to refer to the $i$th row of $A$ and $A[i,j]$ to refer to the $j$th element in the $i$th row.

\subsection{The construction}

Fix $\delta, \eps\in (0,{1\over2}]$, and let
\begin{equation}\label{eq:t}
t=\lceil \eps^{-1}+\log \lceil\log_2 1/\delta\rceil \rceil,
\end{equation}
and
\begin{equation}\label{eq:m}
m=\lceil t^2( M(\chi)+\log_2 |E|) \rceil.
\end{equation} (See Lemma~\ref{lem:sec5:main:induction} for the relation between $\eps$ and $\delta$, and the parameters of Theorem~\ref{thm:allbut}).
For $i\in \mathbb Z$, we first define the map $\Delta_{i}:V\to \mathbb R^{m \times t}$, and then we use it to construct our final embedding.

For a vertex $v\in V$ and $c=\col(v,p(v))$, let
$\alpha=\sum_{c'\in\col(E(P_{v}))} t^2\tau_i(v_{c'})$, and $$\beta=\alpha+\min\left(t^2\tau_i(v),\left\lfloor
{d_T(v_c,v)-\sum_{\ell=-\infty}^{i-1} 2^\ell\tau_\ell(v)\over 2^i/{t^2}}
\right\rfloor\right).$$

Note that $\beta \leq m$ since
$$\tau_i(v) + \sum_{c' \in \chi(E(P_v))} \tau_i(v_c') \leq \varphi(c) \leq M(\chi) + \log_2 |E|\,.$$
For $j \in [m]$, we define,
\begin{equation}\label{eq:delta}
\Delta_{i}(v)[j]=
\left\{
\begin{array}{ll}
\left({2^i\over {t^2}},{0,0\ldots, 0}\right)&\textrm{if $\alpha<{j} \leq \beta$,}\\
\left(d_T(v_c,v)-\left(\left(\sum_{\ell=-\infty}^{i-1} 2^\ell \tau_\ell(v)\right)+(\beta-\alpha){2^i\over {t^2}}\right),{0,0\ldots, 0}\right)&\textrm{if $j=\beta+1$ and $\beta-\alpha< t^2\tau_i(v)$,}\\
({0,0\ldots, 0})&\textrm{otherwise.}
\end{array}
\right.
\end{equation}

\medskip

Observe that the scale selector $\tau_i$ chooses the scales in this definition, and for $v\in V$ and $i\in \mathbb Z$, $\Delta_i(v)=0$ when $\tau_i(v)=0$. Also note that the second case in the definition only occurs when $\tau_i(v)$ is specified by part (A) of \eqref{def:tau}, and in that case $\sum_{\ell \leq i}2^\ell \tau_\ell (v) > d(v,v_c)$.

Now, we present some key properties of the map $\Delta_i(v)$. The following two observations follow
 immediately from the definitions.

 \begin{observation}\label{obv:delta2}
For  $v\in V$ and $i\in \mathbb Z$, each row in $\Delta_i(v)$ has at most one non-zero coordinate.
 \end{observation}

 \begin{observation}\label{obv:delta3}
For $v\in V$ and $i\in \mathbb Z$, let $\alpha=\sum_{c'\in\col(E(P_{v}))} t^2\tau_i(v_{c'})$.
For $j\notin (\alpha ,\alpha+t^2\tau_i(v)]$, we have
$$\Delta_{i}(v)[j]=({0,\ldots,0}) .$$
 \end{observation}

Proofs of the next four lemmas
 will be presented in Section~\ref{subsec:proofs}.

 \begin{lemma}\label{obv:partial}
For  $v\in V$, there is at most one $i\in \mathbb Z$ and at most one couple $(j,k)\in [m]\times [t]$ such that  $\Delta_i(v)[j,k]\notin\{0,{2^i\over t^2}\}$.
 \end{lemma}

\begin{lemma}\label{lem:delta:inc}
Let $c\in \chi(E)$, and $u,w\in V(\gamma_c)\backslash\{v_c\}$ be such that $d_T(w,v_c)\leq d_T(u,v_c)$. For all $i\in \mathbb Z$ and $(j,k)\in [m]\times [t]$, we have
$$\Delta_i(w)[j,k]\leq \Delta_i(u)[j,k].$$
\end{lemma}

\begin{lemma}\label{lem:delta}
For $c\in \chi(E)$, and  $u,w\in V(\gamma_c)\setminus \{v_c\}$, we have
\begin{equation} \label{eq:emb:iso1}
d_T(w,u)=\sum_{i\in \mathbb Z} \|\Delta_i(u)-\Delta_i(w)\|_1,
\end{equation}
and
\begin{equation} \label{eq:emb:iso2}
d_T(v_c,u)=\sum_{i\in \mathbb Z} \|\Delta_i(u)\|_1.
\end{equation}
\end{lemma}

\begin{lemma}\label{lem:new}
For $c\in \chi(E)$,  $u,w\in V(\gamma_c)\setminus \{v_c\}$, $i> j$ and $k \in [m]$, if both $\|\Delta_{i}(u)[k]-\Delta_i(w)[k]\|_1\neq 0$, and $\|\Delta_j(u)[k]-\Delta_j(w)[k]\|_1\neq 0$, then $d_T(u,w)\geq {2^{j-1}}$.
\end{lemma}

\medskip
\noindent
{\bf Re-randomization.}
For $t\in \mathbb N$, let $\pi_t:\mathbb R^t\to \mathbb R^t$ be a random mapping
obtained by uniformly permuting the coordinates in $\mathbb R^t$.
Let $\{\sigma_i\}_{i \in [m]}$ be a sequence of i.i.d. random variables with the same distribution as $\pi_t$.
We define the random variable $\pi_{t,m}:\mathbb R^{m\times t} \to \mathbb R^{m\times t}$ as follows,
\[
\mathcal \pi_{t,m}\left(
\begin{array}{c}
r_1\\ \vdots\\ r_m
\end{array}
\right)
=
\left(
\begin{array}{c}
\sigma_1(r_1)\\
 \vdots\\
 \sigma_m(r_m)
 \end{array}
\right).
\]

\medskip
\noindent
{\bf The construction.}
We now use re-randomization to construct our final embedding. For $c\in \chi(E)$, and $i\in \mathbb Z$,
the map $f_{i,c}: V(T({c}))\to \mathbb R^{m\times t}$ will represent an
embedding of the subtree $T({c})$ at scale $2^i/t^2$.
Recall that,
$$V(T(c))=V(\gamma_c)\cup\left(\bigcup_{c'\in\rho^{-1}(c)}V(T(c'))\setminus\{v_{c'}\} \right).$$

Let $\{\Pi_{i,c'} : i \in \mathbb Z, c' \in \rho^{-1}(c) \}$ be a sequence of i.i.d. random variables which each have the distribution of $\pi_{t,m}$.
We define $f_{i,c}:V(T({c}))\to \mathbb R^{m\times t}$ as follows,
\begin{equation}\label{def:fic}
f_{i,c}(x)=
\left\{
\begin{array}{ll}
0 & \textrm{ if $x=v_c$},\\
\Delta_i(x) & \textrm{ if $x\in V(\gamma_c)\setminus \{v_c\}$},\\
{\Delta_i{(v_{c'})}}+\Pi_{i,c'}( f_{i,c'}(x)) & \textrm{ if $x\in V(T({c'}))\setminus\{v_{c'}\}$ for some $c'\in\rho^{-1}(c)$}.
\end{array}
\right.
\end{equation}

Re-randomization permutes the elements within each row, and the permutations are independent for different subtrees, scales, and rows.
Finally, we define $f_i=f_{i,c_0}$, where $c_0=\chi(r,p(r))$. We use the following lemma to prove Theorem~\ref{thm:allbut}.
\begin{lemma}\label{lem:sec5:main}
There exists a universal constant $C$ such that the following
holds with non-zero probability:
For all $x,y \in V$,
\begin{equation}\label{eq:allbut2}
(1 -C\e) \,d_T(x,y) - \delta\, \rho_{\chi}(x,y;\delta) \leq \sum_{i\in \mathbb Z}\|f_i(x)-f_i(y)\|_1 \leq d_T(x,y)\,.
\end{equation}
\end{lemma}

We will prove Lemma~\ref{lem:sec5:main} in Section~\ref{subsec:embedding}. We first make
two observations, and then use them to prove Theorem~\ref{thm:allbut}.
Our first observation is immediate from Observation~\ref{obv:delta2} and Observation~\ref{obv:delta3}, since in the third case of \eqref{def:fic},  by Observation~\ref{obv:delta3}$,\Delta_i(v_c')$ and $\Pi_{i,c'}( f_{i,c'}(x))$
must be supported on disjoint sets of rows.
\begin{observation}\label{obv:rlpath}
For any $v\in V$ and for any row $j\in[m]$, there is at most one non-zero coordinate in $f_i(v)[j]$.
\end{observation}

Observation~\ref{obv:delta3} and Lemma~\ref{lem:delta} also imply the following.
\begin{observation}\label{lem:rlpath}
For any $v\in V$ and $u\in P_v$, we have
$d_T(u,v)=\sum_{i\in \mathbb Z} \|f_i(u)-f_i(v)\|_1$.
\end{observation}

Using these, together with Corollary~\ref{col:scales}, we now prove Theorem~\ref{thm:allbut}.

\begin{proof}[Proof of Theorem~\ref{thm:allbut}]
By Lemma~\ref{lem:sec5:main}, there exists a choice of mappings $\{g_i\}_{i\in \mathbb Z}$ such that for all $x,y \in V$,
$$d_T(x,y)\geq \sum_{i \in \mathbb Z} \|{g_i(x)-g_i(y)}\|\geq (1-O(\eps))d_T(x,y)-\delta \rho_{\col}(x,y;\delta)\,.$$

We will apply Corollary~\ref{col:scales} to the family given by $\left\{f_i = {t^2g_i\over 2^i}\right\}_{i\in \mathbb Z}$ to arrive at an embedding
$F : V \to \ell_1^{tm\left({2+\left \lceil \log{1\over \eps}\right\rceil}\right)}$ such that $G = F/t^2$ satisfies,
\begin{equation}\label{eq:needthis}
d_T(x,y)\geq \|{G(x)-G(y)}\|_1\geq (1-O(\eps))d_T(x,y)-\delta \rho_{\col}(x,y;\delta).
\end{equation}
  Observe that the codomain of $f_i$ is $\mathbb R^{m\times t}$, where $mt=\Theta( (\frac 1 \eps+\log\log(\frac 1 \delta))^{3}\log n),$ and the codomain of $G$ is $\mathbb R^d$, where $d={\Theta( \log {1\over \eps}(\frac 1 \eps+\log\log(\frac 1 \delta))^{3}\log n)}$.

To achieve \eqref{eq:needthis}, we need only show that for every $x,y \in V$, we have $\zeta(x,y) \lesssim \e d_T(x,y)$,
where $\zeta(x,y)$ is defined in \eqref{eq:zetadef}.
Recalling this definition, we now restate $\zeta$ in terms of our explicit family $\left\{f_i = {t^2 g_i \over 2^i}\right\}_{i \in \mathbb Z}$.
We have,
\begin{equation}\label{eq:zeta:rewrite}
\zeta(x,y)=\hspace{-0.5cm}\sum_{(k_1,k_2)\in [m]\times [t]}\hspace{-0.7cm}\sum_{\substack{i:\exists j<i \\g_{j}(x)[k_1,k_2]\neq g_{j}(y)[k_1,k_2]}} h_i(x,y;k_1,k_2)\,,
\end{equation}
where,
$$h_i(x,y;k_1,k_2)={2^i\over t^2}\left({t^2\over 2^i}\left|g_{i}(x)[k_1,k_2]-g_{i}(y)[k_1,k_2]\right|-\left\lfloor \left|{t^2\over 2^i}g_{i}(x)[k_1,k_2]-{t^2\over 2^i}g_{i}(y)[k_1,k_2]\right|\right\rfloor\right).$$

Fix $x,y \in V$.
For $c\in \chi(E(P_{xy}))$, let $\lambda_{c}$ be the induced subgraph on  $V(P_{xy})\cap V(\gamma_{c})$, i.e. the subpath of $P_{xy}$ where all edges are colored by color $c$.
We have,
\begin{equation}\label{eq:dist:partial}
d_T(x,y)=\sum_{c\in \chi(E(P_{xy}))}\len(E(\lambda_c)).
\end{equation}

If we look at a single term in \eqref{eq:zeta:rewrite}, we have
\begin{align}\label{eq:zeta:term}
h_i(x,y;k_1,k_2)< {2^i\over t^2}.
\end{align}
For $u,v\in P_{xy}$, let $$S_i(u,v)=\{(k_1,k_2)\in [m]\times [t]:h_i(u,v;k_1,k_2)\neq 0 \textrm{ and } \exists j<i: g_{j}(x)[k_1,k_2]\neq g_{j}(y)[k_1,k_2]\}.$$
Now, notice that if $\frac{t^2}{2^i} (g_i(x)[k_1,k_2]-g_i(y)[k_1,k_2])$ is fractional, then there must exist a subpath $\lambda_c$, for a color $c\in \chi (E(P_{xy}))$, with endpoints $u_c$ and $v_c$ such that $\frac{t^2}{2^i} (g_i(u_c)[k_1,k_2]-g_i(v_c)[k_1,k_2])$ is fractional too. Hence we have
$$\zeta(x,y)<\sum_{c\in \chi(E(P_{xy}))} \sum_{i\in \mathbb Z}{2^i|S_i(u_c,v_c)|\over t^2}.$$
We call $\sum_{i\in \mathbb Z} {2^i|S_i(u_c,v_c)|\over t^2}$ the contribution of $\lambda_c$, for each color $c\in \chi (E(P_{xy}))$.


\medskip

We divide the analysis of the paths $\lambda_c$ for ${c\in \chi(E(P_{xy}))}$ into two cases.
For $c\in \chi(E(P_x))\triangle \chi (E(P_y))$, the vertex $v_{c}$  is one endpoint of the path $\lambda_{c}$. Let  $u_{c}$ be the other.
By Lemma~\ref{obv:partial}, there is at most one $i\in \mathbb Z$ and $(k_1,k_2)\in [m]\times [t]$ such that $h_i(u_c,v_c;k_1,k_2)\neq 0,$ and
$$
\left|\bigcup_{i\in \mathbb Z} S_i(u_c,v_c)\right|\leq1
$$
By Lemma~\ref{obv:tau}, for all ${i}\in \mathbb Z$ with ${{d_T(u_{c},v_{c})}} \leq 2^{i-1}$, we have
$\tau_i(u_{c})=0$, and
\begin{equation}\label{sec5:eq:cor3}
\|\Delta_i(u_{c})\|_1=\|g_i(u_{c})-g_i(v_{c})\|_1=0.
\end{equation}
For $i<1+\log_2({{d_T(u_{c},v_{c})}} )$, by \eqref{eq:zeta:term} and Lemma~\ref{obv:partial}
we can bound the contribution of $\lambda_c$ to $\zeta(x,y)$ by,
\begin{equation}\label{eq:tmp:easy}
\sum_{j\in \mathbb Z} {2^j|S_j(u_c,v_c)|\over t^2}< {2^i\over t^2}< {2d_T(u_c,v_c)\over t^2}\leq \eps d_T(u_c,v_c).
\end{equation}

In the case that $c\notin\chi(E(P_x))\triangle \chi(E(P_y))$, note that there is at most one color in $\chi(E(P_{xy}))\setminus(\chi(E(P_x))\triangle \chi(E(P_y)))$. If no such color exists, then by \eqref{eq:tmp:easy},
\[
\zeta(x,y)<  \sum_{c\in \chi(E(P_{xy}))}\eps\len(E(\lambda_c))\overset{\eqref{eq:dist:partial}}{\leq}\eps d_T(x,y).
\]
Suppose now that $\{c\}=\chi(E(P_{xy}))\setminus(\chi(E(P_x))\triangle \chi(E(P_y)))$. Let $u,w \in V(\lambda_c)$ be the closest vertices to $x$
and $y$, respectively. For $i\in \mathbb Z$ we will show that if $h_i(u,w;k_1,k_2)\neq 0$, then either $d_T(x,y)\geq 2^{i-2},$
or for all $j<i$, we have $(g_j(x)-g_j(y))[k_1,k_2]=0$.
 Then, by Observation~\ref{obv:partial}, there are at most two elements in $g_i(u)-g_i(w)$ that are not  in $\{0,{2^i\over t^2},-{2^i\over t^2}\}$, therefore we can conclude
\begin{align*}
\zeta(x,y)&<  \sum_{i\in \mathbb Z}{2^i|S_i(u,w)|\over t^2}+\sum_{c\in  \chi(E(P_x))\triangle \chi (E(P_y))}\sum_{i\in \mathbb Z} {2^i|S_i(u_c,v_c)|\over t^2}\\
&\overset{\eqref{eq:dist:partial}}\leq 4 \eps d_T(x,y)+\sum_{c\in \chi(E(P_x))\triangle \chi (E(P_y))}\eps\,\len(\lambda_c)\\
&\leq 5\eps d_T(x,y).
\end{align*}

Without loss of generality suppose that $d_T(u,v_c)\leq d_T(w,v_c)$. If $d_T(w,v_c)=0$ then the contribution of $\lambda_c$ to $\zeta(x,y)$ is zero. Suppose now that $d_T(w,v_c)>0$, and let $m_w=\max\{i:\tau_i(w)\neq 0\}$. By Lemma~\ref{obv:tau} the maximum always exists.

We will now split the rest of the proof into two cases.

\medskip
\noindent
{\bf Case 1:  $\tau_{m_w-1}(u)=0.$}

\medskip

In this case by Lemma~\ref{lem:tau3} we have $d_T(u,w)> 2^{m_w-1}$. For $(k_1,k_2)\in [m]\times [t]$, if $h_i(u,w;k_1,k_2)\neq 0$ then by \eqref{eq:delta}, $i\leq m_w$ and
\[
{2^{i}\over t^2} \leq {2^{m_w}\over t^2} < {2d_T(u,w)\over t^2}\leq {2d_T(x,y)\over t^2} \leq \e d_T(x,y)\,.
\]

\medskip
\noindent
{\bf Case 2:  $\tau_{m_w-1}(u)\neq 0.$}

\medskip

Let $m_u=\max\{i:\tau_i(u)\neq 0\}$. By Lemma~\ref{lem:incd} and as $\tau_{m_w-1}(u)\neq 0$, we have
$
m_u\leq m_w\leq m_u+1.
$
Observation~\ref{obv:tau:zero2}, implies that for all $j<m_u$,
\begin{equation*}\label{eq:tmp:phi}
\tau_j(u)+\sum_{c'\in \chi(E(P_u))}\tau_j(v_{c'})=\varphi(c).
\end{equation*}
We have $m_w\geq m_u$, and by Observation~\ref{obv:tau:zero2},
\begin{equation}\label{eq:tmp:phi2}
\tau_j(w)+\sum_{c'\in \chi(E(P_w))}\tau_j(v_{c'})=\tau_j(u)+\sum_{c'\in \chi(E(P_u))}\tau_j(v_{c'})=\varphi(c).
\end{equation}
therefore, by Observation~\ref{obv:delta3} for $j<m_u$ and $k\in[t^2\varphi(c)]$
\begin{equation}\label{eq:main:e1}
\|(g_j(x)-g_j(u))[k]\|_1=\|(g_j(y)-g_j(w))[k]\|_1=0,
\end{equation}
and by Observation~\ref{obv:delta3} and part (B) of \eqref{def:tau}, for all $i\in \mathbb Z$, all the non-zero elements of $g_i(u)-g_i(w)$ are in the first $t^2\varphi(c)$ rows.

Suppose that there exists $k\in[m]$ such that $\|(g_i(u)-g_i(w))[k]\|_1\neq 0$. Now, we divide the proof into two cases again.

\medskip
\noindent
{\bf Case 2.1:} There exists a  $j<i$, such that $\|(g_j(x)-g_j(u))[k]\|_1+\|(g_j(y)-g_j(w))[k]\|_1\neq 0.$

\medskip

In this case,  there must exist some $c'\in  \chi(E(P_x))\triangle \chi (E(P_y))$, such that $$\|(g_j(v_{c'})-g_j(u_{c'}))[k]\|_1\neq 0.$$
 By \eqref{def:fic} and \eqref{eq:delta}, we have $\tau_j(u_{c'})\neq 0$. Inequality \eqref{eq:main:e1} implies $j\geq m_u$, and finally by Lemma~\ref{obv:tau},
\begin{equation}\label{eq:tmp:c1}
d_T(x,y)\geq d_T(u_{c'},v_{c'})>  2^{j-1}\geq 2^{m_u-1} \geq 2^{m_w-2}\geq 2^{i-2}.
\end{equation}

\medskip
\noindent
{\bf Case 2.2:} $\|(g_j(x)-g_j(u))[k]\|_1+\|(g_j(y)-g_j(w))[k]\|_1= 0$ for all $j<i$.

\medskip

In this case, either for all $j<i$, $\|g_j(x)[k]-g_j(y)[k]\|_1=0$ which implies that for $k'\in [t]$, $(k,k')\notin S_i(u,w)$, or $\|g_j(u)[k]-g_j(w)[k]\|_1\neq 0$ for some $j<i$. If $\|g_j(u)[k]-g_j(w)[k]\|_1\neq 0$ for some $j<i$ then by Lemma~\ref{lem:new},
\begin{equation}\label{eq:tmp:c2}
d_T(x,y)\geq d_T(u,w)\geq  2^{m_u-1}\geq 2^{m_w-2}\geq 2^{i-2}.
\end{equation}

For $i>m_w$ we have $\|g_i(u)-g_i(w)\|_1=0$,
therefore in both cases if $h_i(x,y;k_1,k_2)\neq 0$ either for all $j<i$, $\|g_j(x)[k]-g_j(y)[k]\|_1=0$ or
$${2^i\over t^2}\leq {4d_T(x,y)\over t^2}\leq 2\eps d_T(x,y).$$
\end{proof}

\subsection{Properties of the $\Delta_i$ maps}\label{subsec:proofs}

We now present proofs of Lemmas \ref{obv:partial}--\ref{lem:new}.

\begin{proof}[Proof of Lemma~\ref{obv:partial}]
For a fixed $i\in \mathbb Z$, by \eqref{eq:delta} there is at most one element in $\Delta_i(v)$ that takes a value other than $\{0,{2^i\over t^2}\}$.

We prove this lemma by showing that if for some $i\in \mathbb Z$, and $(j,k)\in [m]\times [t]$, $$\Delta_i(v)[j,k]\notin
\left\{0,{2^i\over t^2}\right\},$$ then for all $i'>i$ and $(j',k')\in [m]\times [t]$, we have $\Delta_{i'}(v)[j',k']=0$.
Let $c=\col(v,p(v))$. Using \eqref{eq:delta}, we can conclude that
$$t^2\tau_i(v)>\left\lfloor
{d_T(v_c,v)-\sum_{\ell=-\infty}^{i-1} 2^\ell\tau_\ell(v)\over 2^i/{t^2}}\right\rfloor.
$$
Since the left hand side is an integer,
\[t^2\tau_i(v)\geq
{d_T(v_c,v)-\sum_{\ell=-\infty}^{i-1} 2^\ell\tau_\ell(v)\over 2^i/{t^2}},
\]
and
\begin{align*}
\sum_{\ell\leq i}2^\ell \tau_\ell(v)
&=2^i\tau_k(v)+\sum_{\ell< i}2^\ell \tau_\ell(v)\\
&\geq 2^i\left({d_T(v_c,v)-\sum_{\ell<i} 2^\ell\tau_\ell(v)\over 2^i}\right)+\sum_{\ell< i}2^\ell \tau_\ell(v)\\
&\geq d_T(v_c,v).
\end{align*}
By part (A) of \eqref{def:tau}, for $i'>i$ we have $\tau_{i'}(v)=0$, thus $\|\Delta_{i'}(v)\|_1=0$ and the proof is complete.
\end{proof}

\begin{proof}[Proof of Lemma~\ref{lem:delta:inc}]
 For  $i< \left \lfloor \log_2 \left(m(T)\over M(\chi)+\log_2 |E|\right )\right\rfloor$ we have $\|\Delta_k(u)\|=\|\Delta_k(w)\|_1=0$.

 Let $\nu$ be the minimum integer such that part (A) of \eqref{def:tau} for $\tau_\nu(w)$ is less that or equal to part (B). This $\nu$ exists since, by \eqref{eq:tau:b}, part (B) of \eqref{def:tau} is always positive, while by Lemma~\ref{obv:tau}, part (A) of \eqref{def:tau} must be zero for some $\nu \in \mathbb{Z}$. First we analyze the case when $i<\nu$.

Observation~\ref{obv:tau:zero2} implies that part (B) of \eqref{def:tau} specifies the value of $\tau_i(w)$. By Lemma~\ref{lem:incd} $\tau_i(u)\geq \tau_i(w)$, but the part (B) for $\tau_i(u)$ is the same as for $\tau_i(w)$, so we must have $ \tau_i(u)=\tau_i(w),$ and the same reasoning holds for $\tau_{\ell}(w)$ for $\ell<i$. Using this and the fact that part (A) does not define $\tau_i(w)$, we have
 $$
 2^i\tau_i(w)+\sum _{\ell<i}2^\ell\tau_\ell(w)= 2^i\tau_i(u)+\sum _{\ell<i}2^\ell\tau_\ell(u)< d_T(v_c,w) < d_T(v_c,u).
 $$
Therefore, the second case in \eqref{eq:delta} happens neither for $u$ nor for $w$, and for $i<\nu$ we have $\Delta_i(u)=\Delta_i(w)$.

We now consider the case $i=\nu$. We have already shown that for $\ell<i$,  $\tau_\ell(u)=\tau_\ell(w),$ and using \eqref{eq:delta}, it is easy to verify that for all $(j,k)\in [m]\times [t]$,
$$\Delta_i(u)[j,k]\geq \Delta_i(w)[j,k].$$

Finally, in the case that $i>\nu$, by Observation~\ref{obv:tau:zero2}, we have $\tau_i(w)=0$, and $\Delta_i(w)[j,k]=0$.
\end{proof}

\begin{proof}[Proof of Lemma~\ref{lem:delta}]
For all $i\in \mathbb Z$, recalling the definition $\alpha$ and $\beta$ in \eqref{eq:delta} for  $\Delta_i(u)$, we have
$$\beta-\alpha=\min\left(t^2\tau_i(v),\left\lfloor
{d_T(v_c,v)-\sum_{\ell=-\infty}^{i-1} 2^\ell\tau_\ell(v)\over 2^i/{t^2}}
\right\rfloor\right).$$
and by definition of $\Delta_i(u)$ we have,
$$\|\Delta_i(u)\|_1=\min\left({2^i}\tau_i(u), d_T(u,v_c)-\sum_{j<i} 2^j\tau_j(u)\right).$$
By Lemma~\ref{obv:tau2}, we have $\sum_{i\in \mathbb Z} 2^i\tau_i(u)\geq d_T(u,v_c)$, therefore  $d_T(v_c,u)=\sum_{i\in \mathbb Z} \|\Delta_i(u)\|_1.$
The same argument also implies that $d_T(w,v_c)=\sum_{i\in \mathbb Z} \|\Delta_i(w)\|_1$.

Now, suppose that $d_T(u,v_c)\geq d(w,v_c)$. Then Lemma~\ref{lem:delta:inc} implies that, $$\|\Delta_i(u)-\Delta_i(w)\|_1=\|\Delta_i(u)\|_1-\|\Delta_i(w)\|_1=d_T(v_c,u)-d_T(v_c,w)=d_T(w,u).$$
\end{proof}

\begin{proof}[Proof of Lemma~\ref{lem:new}]
Without loss of generality suppose that $d_T(v_c,u)\geq d_T(v_c,w)$. We have,
\begin{align}
d_T(u,w)&= \sum_{h\in \mathbb Z} \|\Delta_h(u)-\Delta_h(w)\|_1\nonumber \\
&\geq  \sum_{h=j}^i \|\Delta_h(u)-\Delta_h(w)\|_1\nonumber \\
&\geq  \|\Delta_i(u)-\Delta_i(w)\|_1+\|\Delta_j(u)-\Delta_j(w)\|_1\,.\label{eq:neweq}
\end{align}
By Lemma~\ref{lem:incd} we have $\tau_j(w)\leq \tau_j(u)$.
If part (B) of \eqref{def:tau} is less than part (A), then by \eqref{eq:delta},
for all $h$ such that
$$\sum_{c'\in\chi(E(P_v))} t^2 \tau_j(v_{c'})<h \leq t^2{\varphi(c)},
$$
we have $\|\Delta_{j}(w)[h]\|_1={2^i\over t^2}$.
And, by Lemma~\ref{lem:delta:inc}, and Observation~\ref{obv:delta3} for $k\in \mathbb Z$, $\Delta_{j}(w)=\Delta_{j}(u)$. Hence, part (A) of \eqref{def:tau} must specify the value of $\tau_j(w)$. Observation~\ref{obv:tau:zero2} implies that $\tau_i(w)=0$ and by \eqref{eq:delta}, we have $\|\Delta_i(w)\|_1=0$.

By \eqref{eq:delta}, since $\|\Delta_{i}(u)[k]\|_1>0$, and $\alpha$ from \eqref{eq:delta} is a multiple of $t^2$, for all $t^2 \lfloor{k\over t^2}\rfloor<h<k$ we have $\|\Delta_{i}(u)[h]\|_1={2^i\over t^2}$.
This implies that,
$$\|\Delta_i(u)-\Delta_i(w)\|_1\geq {2^i\over t^2} \left(k-1-t^2\left\lfloor k\over {t^2}\right\rfloor\right)\geq {2^j\over t^2}\left( k-1-t^2\left\lfloor k \over {t^2}\right\rfloor\right).$$
 Moreover, $\|\Delta_{j}(w)[k]\|_1<{2^j\over t^2}$, and \eqref{eq:delta} implies that for all $k<h\leq t^2 \lfloor{1+{k\over t^2}}\rfloor$, we have $\|\Delta_{j}(w)[h]\|_1=0$. The same argument also shows that,
$$\|\Delta_j(u)-\Delta_j(w)\|_1\geq {2^j\over t^2}\left( t^2\left\lfloor1+{k\over {t^2}}\right\rfloor-k\right).$$
Hence by \eqref{eq:neweq},
$$d_T(u,w)\geq {t^2-1\over t^2}2^j\geq 2^{j-1}.$$
\end{proof}

\subsection{The probabilistic analysis}\label{subsec:embedding}

We are thus left to prove Lemma \ref{lem:sec5:main}.
For $c\in \chi(E)$, we analyze the embedding for $T(c)$ by going through all  $c'\in \chi(E(T(c)))$ one by one in increasing order of $\varphi(c')$.
Our first lemma bounds the probability of a bad event, i.e. of a subpath not contributing enough to the distance
in the embedding.

\begin{lemma}\label{lem:con}
For any $C \geq 8$, the following holds.
Consider three colors $a \in \chi(E)$, $b \in \rho^{-1}(a)$, and $c \in \chi(E(P_{u\, v_b}))$
for some $u \in V(T(b))$.  Then for every $w \in V(T(a)) \setminus V(T(b))$, we have
\begin{align}\label{eq:con2}
&\pr \left[ \exists\,{x\in V(P_{w\,v_{a}})} :\sum_{i\in \mathbb Z}\|f_{i,a}(x)-f_{i,a}(u)\|_1 \leq \left(1-{C \eps}\right)\,d_T(u, v_c)+\sum_{i\in \mathbb Z}\|f_{i,a}(v_c)-f_{i,a}(x)\|_1 \mid  \{f_{i,c'}\}_{c'\in \rho^{-1}(a)} \right]\nonumber
\\&
\hspace{3.2in} \leq  {1\over \left\lceil\log_2 1/\delta\right\rceil}  \exp\left(-(C/(\eps 2^{\beta+2}))\, d_T(u,v_c)\right),
\end{align}
where $\beta=\max\{i: \exists y \in P_{u\,v_c} \backslash \{v_c\}, \tau_i(y)\neq 0\}$. (See Figure~\ref{fig:con} for position of vertices in the tree.)
\end{lemma}
\begin{figure}[htbp]
\begin{center}
  \includegraphics[width=7cm]{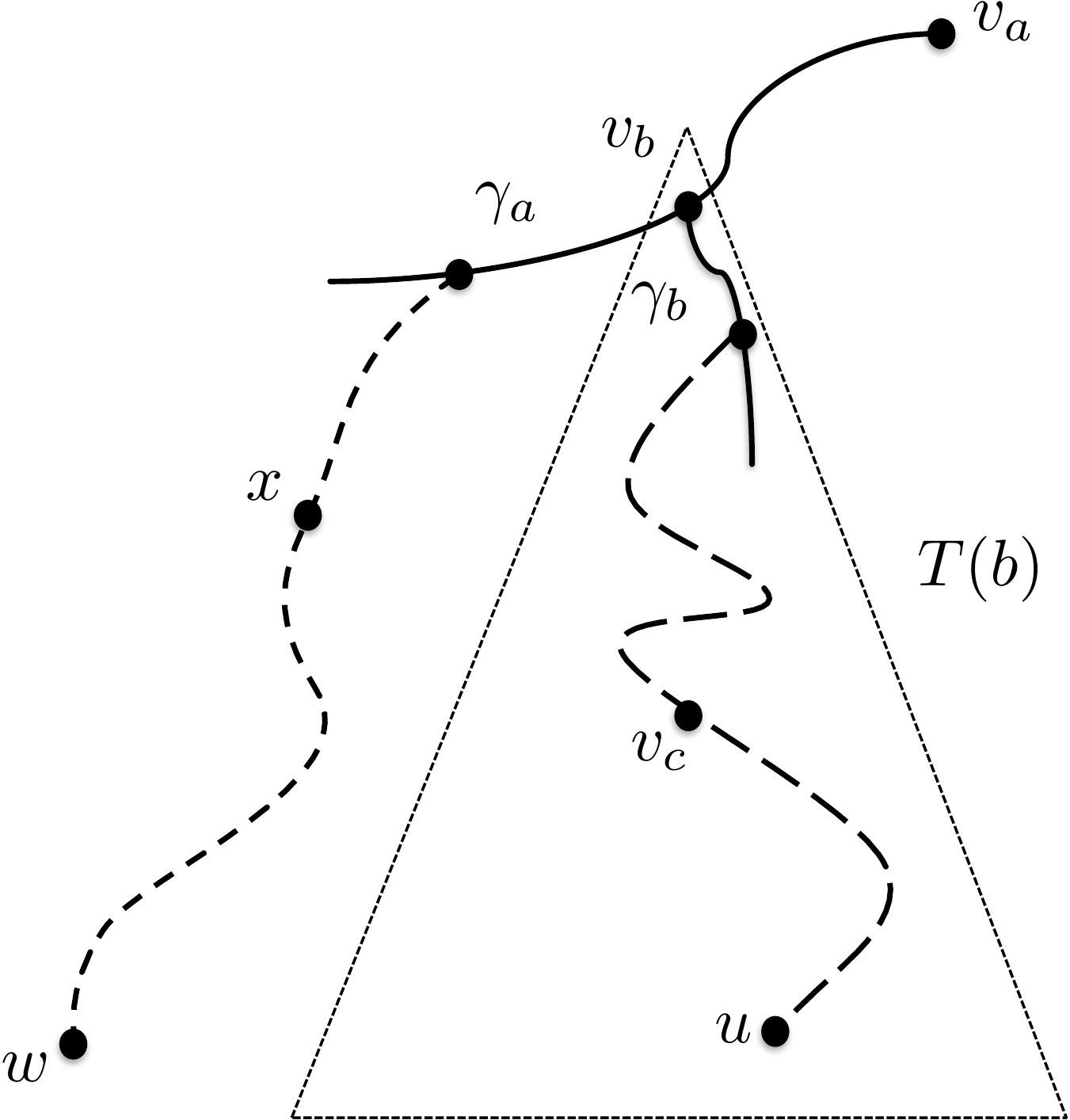}
  \caption{{ Position of vertices corresponding to the statement of Lemma~\ref{lem:con}. }}
  \label{fig:con}
\end{center}
\end{figure}

\begin{proof}
Recall that $\mathbb R^{m\times t}$ is the codomain of $f_{i,a}$.
For $i\in \mathbb Z$, and $j\in[m]$, and $z\in V(P_{w\,v_{a}})$, let
$$s_{ij}(z)=\left\|\vphantom{\bigoplus} f_{i,a}(z)[j]-f_{i,a}(v_c)[j]\right\|_1+\left\|\vphantom{\bigoplus} f_{i,a}(v_c)[j]-f_{i,a}(u)[j]\right\|_1-\left\|\vphantom{\bigoplus}  f_{i,a}(z)[j]-f_{i,a}(u)[j]\right\|_1.$$
We have,
\begin{equation*}
\sum_{i\in \mathbb Z} \|f_{i,a}(u)-f_{i,a}(v_c)\|_1+\sum_{i\in \mathbb Z} \|f_{i,a}(v_c)-f_{i,a}(z)\|_1=
\sum_{i\in \mathbb Z} \|f_{i,a}(z)-f_{i,a}(u)\|_1+\sum_{i\in \mathbb Z,j\in [m]} s_{ij}(z).
\end{equation*}
By Observation~\ref{lem:rlpath}, we have $d_T(u,v_c)=\sum_{i\in \mathbb Z} \|f_{i,a}(u)-f_{i,a}(v_c)\|_1$, therefore
\begin{equation}\label{eq:endproof}
d_T(u,v_c)-\sum_{i\in \mathbb Z,j\in [m]} s_{ij}(z)=
\sum_{i\in \mathbb Z} \|f_{i,a}(z)-f_{i,a}(u)\|_1-\sum_{i\in \mathbb Z} \|f_{i,a}(z)-f_{i,a}(v_c)\|_1.
\end{equation}
Let $\mathcal E = \{f_{i,c'} : c' \in \rho^{-1}(a)\}$.  We define $\pr_{\mathcal E}[\cdot] = \pr[\cdot\mid\mathcal E].$
In order to prove this theorem, we bound
$${\pr}_{\mathcal E}\left[\exists\, {x\in V(P_{w\,v_{a}})} : \sum_{i\in \mathbb Z,j\in [m]} s_{ij}(x)\geq C\eps d_T(u,v_c)\right].$$
We start by bounding the maximum of the random variables $s_{ij}$.

For $i>\beta$ we have $\Delta_i(u)=\Delta_i(v_c)$, hence $f_{i,a}(u)=f_{i,a}(v_c)$. Using the triangle inequality for all  for all $i\in Z$, $j\in [m]$ and $z\in P_{w\,v_a}$,
\begin{equation}\label{tmp:tau:max}
s_{ij}(z)\leq 2 \|f_{i,a}(v_c)[j]-f_{i,a}(u)[j]\|_1,
\end{equation}
Hence for all $i\in Z$ and $j\in [m]$ by Observation~\ref{obv:rlpath},
\begin{equation}\label{tmp:tau:max2}
s_{ij}(z)\leq 2 \|f_{i,a}(v_c)[j]-f_{i,a}(u)[j]\|_1\leq {2^{\beta+1}\over t^2}.
\end{equation}
First note that, if $z$ is on the path between $v_{b}$ and $v_a$ then by Observation~\ref{lem:rlpath}, $s_{ij}(z)=0$.
Observation~\ref{obv:delta3} and \eqref{eq:delta} imply that if $\|f_{i,a}(u)[j]-f_{i,a}(v_c)[j]\|_1\neq 0$ then $\|f_{i,a}(v_c)[j]\|_1= 0$.
 From this, we can conclude that $s_{ij}(z)\neq 0$ if and only if there exists a $k\in [t]$ such that both $f_{i,a}(u)[j,k]-f_{i,a}(v_c)[j,k]\neq 0$ and $f_{i,a}(z)[j,k]\neq0$. Since by Lemma~\ref{lem:delta:inc}, for all $i\in \mathbb Z$, $j\in [m]$ and $k\in [t]$, we have $f_{i,a}(w)[j,k]\geq f_{i,a}(z)[j,k]$, we conclude that for $z\in P_{w\,v_{a}}$ if $s_{ij}(z)\neq 0$ then $s_{ij}(w)\neq 0$.

Now, for $i\in \mathbb Z$ and $j\in [m]$, we define a random variable
\begin{equation}
X_{ij}=\begin{cases}
0& \textrm{if $s_{ij}(w)=0$,}\\
2\|f_{i,a}(u)[j]-f_{i,a}(v_c)[j]\|_1& \textrm{if $s_{ij}(w)\neq 0$.}
\end{cases}
\end{equation}
Note that since the re-randomization in \eqref{def:fic} is performed independently on each row and at each scale,
the random variables $\left\{X_{ij}: i\in \mathbb Z, j\in [m]\right\}$ are mutually independent. By \eqref{tmp:tau:max}, for all $z\in P_{w\,v_a}$, we have
$s_{ij}(z)\leq X_{ij}$, and thus
\begin{equation}\label{eq:con2:sleqx}
{\pr}_{\mathcal E}\left[\exists\, {x\in V(P_{w\,v_{a}})} : \sum_{i\in \mathbb Z,j\in [m]} s_{ij}(x)\geq C\eps d_T(u,v_c)\right]\leq {\pr}_{\mathcal E}\left[\sum_{i\in \mathbb Z,j\in [m]} X_{ij}\geq C\eps d_T(u,v_c)\right].
\end{equation}

As before, for $X_{ij}$ to be non-zero, it must be that $k\in [t]$ is such that $f_{i,a}(w)[j,k] \neq 0$ and $f_{i,a}(u)[j,k]-f_{i,a}(v_c)[j,k]\neq 0$. Since $w\notin V(T(b))$ with the re-randomization in \eqref{def:fic} and Observation~\ref{obv:rlpath}, this happens at most with probability $\frac{1}{t}$, hence for $j\in [m]$, and $i\in \mathbb Z$,
\begin{align*}{\pr}_{\mathcal E}[ & X_{ij}\neq 0] \\
&={\pr}_{\mathcal E}\big[\|f_{i,a}(w)[j]-f_{i,a}(v_c)[j]\|_1+\|f_{i,a}(v_c)[j]-f_{i,a}(u)[j]\|_1-\|f_{i,a}(w)[j]-f_{i,a}(u)[j]\|_1\neq 0\big] \\
&\leq {1\over t}.
\end{align*}
This yields,
\begin{equation}\label{eq:ij:expected}
\E[X_{ij}\mid \mathcal E]\leq {1\over t}\left(2\|f_{i,a}(u)[j]-f_{i,a}(v_c)[j]\|_1\right).
\end{equation}
Now we use \eqref{tmp:tau:max2} to write
\[
\Var(X_{ij}\mid \mathcal E)\leq  {1\over t}{\left(2\|f_{i,a}(u)[j]-f_{i,a}(v_c)[j]\|_1\right)^2}\leq {2^{\beta+2}\over t^3}\|f_{i,a}(u)[j]-f_{i,a}(v_c)[j]\|_1,
\]
and use Observation~\ref{lem:rlpath} in conjunction with \eqref{eq:ij:expected} to conclude that
\begin{equation}\label{eq:con2:exp}
\E\left[\sum_{i\in \mathbb Z,j\in [m]} X_{ij}\mid \mathcal E\right]\leq \sum_{i\in \mathbb Z,j\in[m]} {2\over t}\, \|f_i(v_c)[j]-f_i(u)[j]\|_1={2\over t}\, d_T(v_c,u),
\end{equation}
and
\begin{equation}\label{eq:con2:var}
\sum_{i\in \mathbb Z,j\in [m]}\Var(X_{ij}\mid \mathcal E)\leq  \sum_{i\in \mathbb Z,j\in[m]}{2^{\beta+2}\over t^3}\|f_i(v_c)[j]-f_i(u)[j]\|_1={2^{\beta+2}\over t^3}d_T(v_c,u).
\end{equation}
Define $M = \max \{ X_{ij} - \E[X_{ij} \mid \mathcal E] : i\in\mathbb Z, j\in [m] \}.$ We now apply Theorem~\ref{thm:azuma} to complete the proof:
\begin{align*}
{\pr}_{\mathcal E}\Bigg[\sum_{i\in \mathbb Z,j\in [m]} X_{ij}&\geq C\left({d_T(u,v_c)\over t}\right)\Bigg]\\
&={\pr}_{\mathcal E}\Bigg[\sum_{i\in \mathbb Z,j\in [m]} X_{ij}-{2d_T(u,v_c)\over t}\geq (C-2)\left({d_T(u,v_c)\over t}\right)\Bigg]\\
&\overset{\eqref{eq:con2:exp}}{\leq}
{\pr}_{\mathcal E}\left[\sum_{i\in \mathbb Z,j\in [m]} {X_{ij} }-\mathbb E\left[\sum_{i\in \mathbb Z,j\in [m]} X_{ij}\mid \mathcal E\right] \geq (C-2)\left(d_T(u,v_c)\over t\right)\right]\\
&\leq\exp\left({-((C-2)d_T(u,v_c)/t)^2\over 2\left(\sum_{i\in \mathbb Z,j\in [m]} \Var(X_{ij}\mid\mathcal E)+ (C-2)(d_T(u,v_c)/t) M/3\right)}\right).
\end{align*}
Since $\E[X_{ij}\mid\mathcal E]\geq 0$, \eqref{tmp:tau:max2} implies $M\leq {2^{\beta+1}\over t^2}$. Now, we can plug in this bound and \eqref{eq:con2:var} to write,
\begin{align*}
{\pr}_{\mathcal E}\Bigg[\sum_{i\in \mathbb Z,j\in [m]} X_{ij}&\geq C\left({d_T(u,v_c)\over t}\right)\Bigg]\\
&\leq
\exp\left({-((C-2)d_T(u,v_c)/t)^2\over2\left({2^{\beta+2}\over t^3}d_T(u,v_c)+ (C-2)(d_T(u,v_c)/t) (2^{\beta+1}/t^2)/3\right)}\right)\\
&=\exp\left({-t(C-2)^2d_T(u,v_c)\over{2\left(2^{\beta+2}+ (C-2)(2^{\beta+1})/3\right)}}\right)\\
&=\exp\left({-(C-2)^2 \over {(C-2)/3+2}}\left({td_T(u,v_c)\over 2^{\beta+2}}\right)\right).
\end{align*}

An elementary calculation shows that for $C\geq 8$, ${(C-2)^2 \over {(C-2)/3+2}}> C,$ hence
\begin{align*}
{\pr}_{\mathcal E}\Bigg[\sum_{i\in \mathbb Z,j\in [m]} X_{ij}\geq C\left({d_T(u,v_c)\over t}\right)\Bigg]&<  \exp\left(-(Ct/ 2^{\beta+2})\, d_T(u,v_c)\right)\\
&\overset{\eqref{eq:t}}{\leq}\exp\left(-C\left({1\over \eps}+\log\left\lceil\log_2 {1\over \delta}\right\rceil\right)\left({1\over  2^{\beta+2}}\right)\, d_T(u,v_c)\right)\\
&=  \left({1\over \left\lceil\log_2(1/\delta)\right\rceil}\right)^{{Cd_T(u,v_c)\over  2^{\beta+2}}} \cdot \exp\left(-C\left({1\over \eps}\right)\left({1\over  2^{\beta+2}}\right)\, d_T(u,v_c)\right).
\end{align*}

 Since  there exists a $y \in P_{u\,v_c} \backslash \{v_c\}$ such that $\tau_\beta(y) \neq 0$, and for all $c'\in \chi(E)$, $\kappa(c')\geq 1$, Lemma~\ref{obv:tau} implies that $d_T(u,v_c)>  2^{\beta-1}$, and for $C\geq 8$, we have ${Cd_T(u,v_c)\over  2^{\beta+2}}> 1$.  Therefore,
 \begin{align*}
 &\pr_{\mathcal E} \left[ \exists\,{x\in V(P_{w\,v_{a}})} :\sum_{i\in \mathbb Z}\|f_{i,a}(x)-f_{i,a}(u)\|_1 \leq \left(1-{C \eps}\right)\,d_T(u, v_c)+\sum_{i\in \mathbb Z}\|f_{i,a}(v_c)-f_{i,a}(x)\|_1 \right]
  \\&\hspace{3in}\overset{\eqref{eq:endproof}}{\leq}
 {\pr}_{\mathcal E}\left[\exists\, {x\in V(P_{w\,v_{c}})} : \sum_{i\in \mathbb Z,j\in [m]} s_{ij}(x)\geq C\eps d_T(u,v_c)\right]
 \\&\hspace{3in}\overset{\eqref{eq:con2:sleqx}}{\leq}
{\pr}_{\mathcal E}\Bigg[\sum_{i\in \mathbb Z,j\in [m]} X_{ij}\geq C\eps \left({d_T(u,v_c)}\right)\Bigg]
\\&\hspace{3in}\overset{\eqref{eq:t}}{\leq}
{\pr}_{\mathcal E}\Bigg[\sum_{i\in \mathbb Z,j\in [m]} X_{ij}\geq C\left({d_T(u,v_c)\over t}\right)\Bigg]
\\&\hspace{3in}<   \left({1\over \left\lceil\log_2(1/\delta)\right\rceil}\right)\cdot \exp\left(-C\left({1\over \eps 2^{\beta+2}}\right)\, d_T(u,v_c)\right),
 \end{align*}
completing the proof.
 \end{proof}

\medskip
\noindent
{\bf The $\Gamma_a$ mappings.}
Before proving Lemma~\ref{lem:sec5:main}, we need some more definitions.
For a color $a\in \chi(E)$, we define a map $\Gamma_{a}:{V(T(a))}\to V(T(a))$ based on Lemma~\ref{lem:con}. For $u\in V(\gamma_a)$, we put $\Gamma_a(u)=u$.
For all other vertices $u\in V(T(a))\setminus V(\gamma_a)$, there exists a unique color $b\in \rho^{-1}(a)$ such that $u\in V(T(b))$.
We define $\Gamma_{a}(u)$ as the vertex $w \in V(P_{u v_b})$ which is closest to the root
among those vertices satisfying the following condition:
For all $v \in V(P_{u w}) \setminus \{w\}$ and $k \in \mathbb Z$, $\tau_k(v)\neq 0$ implies
\begin{equation}
2^k< {d_T(u,w)\over  \eps (\varphi(\chi(u,p(u)))-\varphi(a))}.
\label{def:gamma:last}
\end{equation}
Clearly such a vertex exists, because the conditions
are vacuously satisfied for $w=u$.
We now prove some properties of the map $\Gamma_a$.

\remove{
Later in the proof of Lemma~\ref{lem:sec5:main:induction}, we will apply Lemma~\ref{lem:con} with $\Gamma_a(u)$ as $v_c$, where $c$ is a color satisfying the conditions of  Lemma~\ref{lem:con}, i.e., $c\in \chi(E(P_{u\,v_a}))\setminus\{a\}$.
The following lemma explains why we are allowed to do that.}

\begin{lemma}\label{obv:gamma1}
Consider any $a\in \chi(E)$ and $u\in V(T(a))$ such that $\Gamma_{a}(u)\neq u$.
Then we have $\Gamma_{a}(u)=v_{c}$ for some $c\in \chi(E(P_{u v_a}))\setminus\{a\}$.
 \end{lemma}
 \begin{proof}
Let $w\in V(P_{u\,\Gamma_a(u)})$ be such that $\Gamma_a(u)=p(w)$.
The vertex $w$ always exists because $\Gamma_a(u)\in V(P_u)\setminus \{u\}$.
 If $\chi(w,\Gamma_a(u))\neq \chi(\Gamma_a(u),p(\Gamma_a(u)))$ then $\Gamma_a(u)$ is $v_{c}$ for some $c\in \chi(E(P_{u\, v_a}))\setminus \{a\}$.

Now, for the sake of contradiction suppose that $\chi(w,\Gamma_a(u))= \chi(\Gamma_a(u),p(\Gamma_a(u)))$. In this case, we show that for all $ v \in P_{u\, p(\Gamma_a(u))} \setminus \{p(\Gamma_a(u))\}$, and $k\in \mathbb Z$, $\tau_k(v)\neq 0$ implies
\begin{equation}\label{eq:obv:gamma1:tv1}
2^k< {d_T(u,p(\Gamma_a(u)))\over  \eps (\varphi(\chi(u,p(u)))-\varphi(a))}.
\end{equation}
This is a contradiction since by definition of $\Gamma_a$, it must be that $\Gamma_a(u)$ is the closest vertex to the root satisfying this condition,
 yet $p(\Gamma_a(u))$ is closer to root than $\Gamma_a(u)$.

Observe that,
$$V(P_{u\,p(\Gamma_a(u))})\setminus\{p(\Gamma_a(u))\}= V(P_{u\,\Gamma_a(u)})\,.$$
We first verify \eqref{eq:obv:gamma1:tv1} for $\Gamma_a(u)$ and $k\in \mathbb Z$ with $\tau_k(\Gamma_a(u))\neq 0$.
Since $\Gamma_a(u)\in V(P_u)$, we have
\begin{equation} \label{eq:gamma:trivial}
d_T(u,\Gamma_a(u))\leq d_T(u,p(\Gamma_a(u))).
\end{equation}
Recalling that $p(w)=\Gamma_a(u)$,
by Lemma~\ref{lem:incd} for all $k\in \mathbb Z$, $\tau_k(\Gamma_a(u))\leq \tau_k(w)$, therefore for all $k\in \mathbb Z$, with $\tau_k(\Gamma_a(u))\neq0$, we have
$\tau_k(w) \neq 0$ as well, hence \eqref{def:gamma:last} implies
\begin{equation}\label{eq:obv:gamma1:2}
2^k<{d_T(u,\Gamma_a(u))\over  \eps (\varphi(\chi(u,p(u)))-\varphi(a))}\overset{\eqref{eq:gamma:trivial}}{\leq} {d_T(u,p(\Gamma_a(u)))\over  \eps (\varphi(\chi(u,p(u))-\varphi(a))}.
\end{equation}

For all other vertices, $v \in V(P_{u\Gamma_a(u)})\setminus\{\Gamma_a(u)\}$, and $k\in \mathbb Z$ with $\tau_k(v)\neq 0$ by \eqref{def:gamma:last},
\begin{equation}
2^k< {d_T(u,\Gamma_a(u))\over  \eps (\varphi(\chi(u,p(u)))-\varphi(a))} \overset{\eqref{eq:gamma:trivial}}{\leq} {d_T(u,p(\Gamma_a(u)))\over  \eps (\varphi(\chi(u,p(u)))-\varphi(a))},
\end{equation}
completing the proof.
 \end{proof}

\begin{lemma}\label{obv:gamma}
Suppose that $a\in \chi(E)$ and $u\in V(T(a))$. For any  $w\in V(P_{u\,\Gamma_{a}(u)})$, such that $\col(u,p(u))=\col(w,p(w))$ we have $\Gamma_{a}(w)\in V(P_{u\,\Gamma_{a}(u)})$.
\end{lemma}
\begin{proof}
For the sake of contradiction, suppose that $\Gamma_{a}(w)\notin V(P_{u\,\Gamma_{a}(u)})$. Since $w\in V(P_u)$, and  $\Gamma_{a}(w)\notin V(P_{u\,\Gamma_{a}(u)})$, we have $\Gamma_a(w)\in V(P_{\Gamma_a(u)})$, and
\begin{equation} \label{eq:gamma:trivial2}
d_T(u,\Gamma_a(u))\leq d_T(u,\Gamma_a(w)).
\end{equation}
Since $w\in V(P_{u\,\Gamma_{a}(u)})$ by assumption, for all vertices, we have $V(P_{u\,w})\setminus \{w\} \subseteq V(P_{u\,\Gamma_{a}(u)})\setminus\{\Gamma_a(u)\}$.
Thus for all $v\in V(P_{u\,w})\setminus \{w\}$ and $k\in \mathbb Z$ with $\tau_k(v)\neq 0$ by \eqref{def:gamma:last},
\begin{equation}\label{eq:gamma:tv1}
2^k<  {d_T(u,\Gamma_a(u))\over  \eps (\varphi(\chi(u,p(u)))-\varphi(a))} \overset{\eqref{eq:gamma:trivial2}}\leq {d_T(u,\Gamma_a(w))\over  \eps (\varphi(\chi(u,p(u)))-\varphi(a))}.
\end{equation}

The fact that $w\in V(P_{u\,\Gamma_a(u)})$ also implies that  $d_T(w,\Gamma_a(w)))\leq d_T(u\,\Gamma_a(w)))$. Therefore, for all vertices
$v\in V(P_{w\,\Gamma_a(w)})\setminus \{\Gamma_a(w)\}$ and $k\in \mathbb Z$ with $\tau_k(v)\neq0$ by \eqref{def:gamma:last},
\begin{equation}\label{eq:gamma:tv2}
2^k< {d_T(w,\Gamma_a(w))\over  \eps (\varphi(\chi(w,p(w)))-\varphi(a))}\leq {d_T(u,\Gamma_a(w))\over  \eps (\varphi(\chi(w,p(w)))-\varphi(a))}= {d_T(u,\Gamma_a(w))\over  \eps (\varphi(\chi(u,p(u)))-\varphi(a))}.
\end{equation}

We have,
\[V(P_{u\,\Gamma_a(w)})=V(P_{u\,w}) \cup \left(V(P_{w\,\Gamma_a(w)})\setminus \{\Gamma_a(w)\}\right). \]
Hence, by \eqref{eq:gamma:tv1} and \eqref{eq:gamma:tv2}, for all $v\in V(P_{u\,\Gamma_a(w)})\setminus \{\Gamma_a(w)\}$ and $k\in \mathbb Z$, $\tau_k(v)\neq 0$ implies
\begin{equation}
2^k< {d_T(u,p(\Gamma_a(w)))\over  \eps (\varphi(\chi(u,p(u)))-\varphi(a))}.
\end{equation}
This is a contradiction to the definition of $\Gamma_a(u)$, since $\Gamma_a(u)$ must be the closest vertex to the root satisfying this condition,
yet $\Gamma_a(w)$ is closer to root than $\Gamma_a(u)$.
\end{proof}

\medskip
\noindent
{\bf Defining representatives for $\gamma_c$.}
Now, for each $c\in \chi(E)$, we define a small set of representatives for vertices in $\gamma_c$. Later, we use these sets to bound the contraction of pairs of vertices that have one endpoint in $\gamma_c$.

For $a\in \chi(E)$ and $c\in \chi(E(T(a)))\setminus \{a\}$, we define the set $R_{a}(c)\subseteq V(\gamma_c)$, the set of representatives for $\gamma_c$, as follows
\begin{align}
 R_{a}(c)=
 \bigcup_{i=0}^{ \lceil\log_2 {1\over \delta}\rceil-1} \Big\{u\in V(\gamma_c):\, & \textrm{$u$ is  the furthest vertex} \nonumber \\
 & \textrm{from $v_c$ s.t. $\Gamma_a(u)\neq u$ and $d(u,v_c)\leq 2^{-i}\,\len(\gamma_c)$}\Big\}. \label{eq:def:Ri}
\end{align}

The next lemma says when a vertex has a close representative.

\begin{lemma}\label{obv:rcdelta}
Consider $a\in \chi (E)$ and $c\in \chi(E(T(a)))\setminus \{a\}$.
For all vertices $u\in V(\gamma_c)$ with  $\Gamma_{a}(u)\neq u$ there exists a $w\in R_{a}(c)$ such that,
\[
d_T(u,v_{c})\leq d_T(w,v_{c})\leq 2\max\big(d_T(u,v_{c}), \delta\, \len(\gamma_c)\big).
\]
\end{lemma}
\begin{proof}
Let $i\geq 0$ be such that $$\frac{d_T(u,v_{c})}{\len(\gamma_c)} \in \left(2^{-i-1},2^{-i}\right]\,.$$ If $i\leq  \lceil\log_2 {1\over \delta}\rceil-1$, then $\eqref{eq:def:Ri}$ implies that either $u\in R_a(c)$, or there exists a $w\in R_a(c)$ such that
$$d_T(u,v_{c})< d_T(w,v_{c})\leq {\len(\gamma_c)\over 2^i}\leq 2\,d_T(u,v_{c}).$$

On the other hand,
if $i>  \lceil\log_2 {1\over \delta}\rceil-1$,  then $\eqref{eq:def:Ri}$ implies that either $u\in R_a(c)$, or that there exists a $w\in R_a(c)$, such that
$$d_T(u,v_{c})< d_T(w,v_{c})\leq {\len(\gamma_c)\over 2^{\lceil \log_2 {1\over \delta}\rceil-1}} \leq 2\delta  \,\len(\gamma_c),$$
completing the proof.
\end{proof}

The following lemma, in conjunction with Lemma~\ref{obv:rcdelta}, reduces the number of vertices in $V(\gamma_c)$ that we need to analyze using Lemma~\ref{lem:con}.

\begin{lemma}\label{lem:correlation}
Let $(X,d)$ be a pseudometric,  and let  $f:V\to X$ be a $1$-Lipschitz map. For $x,y\in V$, and $x',y'\in V(P_{xy})$ and $h\geq 0$, if $d(f(x),f(y))\geq d_T(x,y)-h$ then $d(f(x'),f(y'))\geq d_T(x',y')-h$.
\end{lemma}

\begin{proof}
Suppose without loss of generality that $d_T(x',x)\leq d_T(y',x)$. Using the triangle inequality,
\begin{align*}
d(f(x'),f(y'))&\geq d(f(x),f(y))- d(f(x),f(x'))- d(f(y),f(y'))\\
&\geq (d_T(x,y)-h) - d(f(x),f(x'))- d(f(y),f(y'))\\
&\geq d_T(x,y)- d_T(x,x')- d_T(y,y')-h\\
&= d_T(x',y')-h\,.
\end{align*}
\end{proof}

The following lemma constitutes the inductive step of the proof of Lemma~\ref{lem:sec5:main}.

\begin{lemma}\label{lem:sec5:main:induction}
There exists a universal constant $C$, such that for any color $c\in \chi(E)\cup\{\chi(r,p(r))\}$, the following holds.
Suppose that, with non-zero probability, for all $c'\in \rho^{-1}(c)$, and for all pairs $x,y\in V(T(c'))$, we have
\begin{equation}\label{eq:15:hypo}
(1 -C\e) \,d_T(x,y) - \delta\, \rho_{\chi}(x,y;\delta) \leq \sum_{i\in \mathbb Z}\|f_{i,c'}(x)-f_{i,c'}(y)\|_1 \leq d_T(x,y)\,.
\end{equation}
Then with non-zero probability for all $x,y\in V(T(c))$, we have
\begin{equation}\label{eq:15:claim}
(1 -C\e) \,d_T(x,y) - \delta\, \rho_{\chi}(x,y;\delta) \leq \sum_{i\in \mathbb Z}\|f_{i,c}(x)-f_{i,c}(y)\|_1 \leq d_T(x,y)\,.
\end{equation}
\end{lemma}
\begin{proof}
Let $\mathcal E$ denote the event that,
for all $c'\in \rho^{-1}(c)$, and all $x,y \in V(T(c'))$, we have
\begin{equation}\label{eq:15:hypothesis}
d_T(x,y)\geq \sum_{i \in \mathbb Z} \|{f_{i,c'}(x)-f_{i,c'}(y)}\|\geq (1-C\eps)d_T(x,y)-\delta \rho_{\col}(x,y;\delta)\,.
\end{equation}
We will prove the lemma by showing that, conditioned on $\mathcal E$, \eqref{eq:15:claim} holds with non-zero probability.

For $x,y\in V(T(c))$ we define,
\[
\mu(x,y) = \max \{ \varphi(a) : a \in \chi(E) \textrm{ and } x,y \in V(T(a)) \}\,.
\]
Note that since $x,y\in V(T(c))$, we have
\begin{equation}\label{eq:mu:max}
\mu(x,y)\geq \varphi(c)\,.
\end{equation}

It is easy to see that if $\mu(x,y)> \varphi(c)$, then $x,y\in V(T(c'))$ for some $c'\in \rho^{-1}(c)$.
By construction, if $c' \in \rho^{-1}(c)$ and $x,y \in V(T(c'))$, then
$$
\|f_{i,c}(x)-f_{i,c}(y)\| = \|f_{i,c'}(x)-f_{i,c'}(y)\|,
$$
hence $\mathcal E$ implies that \eqref{eq:15:claim} holds for all such pairs.
Thus in the remainder of the proof, we need only handle pairs $x,y \in V(T(c))$ with $\mu(x,y)=\varphi(c)$.

\medskip

Write $\chi(E(T(c))) = \{c_1, c_2, \ldots, c_n\}$, where the colors are ordered so that $\varphi(c_j) \leq \varphi(c_{j+1})$ for $j=1,2,\ldots,n-1$.
Let $\eps_1=24 \eps$, where the constant $24$ comes from Lemma~\ref{lem:con}. And let $\eps_2=2 \cdot C' \eps$, where $C'$ is the constant from Lemma~\ref{lem:triangle}.

For $i\in[m]$, we define the event $X_i$ as follows:  For all $j \leq i$, and all $x \in V(\gamma_{c_i})$ and $y \in V(\gamma_{c_j})$ with $\mu(x,y)=\varphi(c)$, we have
\begin{align}
\sum_{k\in \mathbb Z}\|f_{k,c}(x)-f_{k,c}(y)\|_1\geq d_T(x,y)-\eps_1 d_T(x,y)
-\eps_2 d_T(\Gamma_c(x), \Gamma_c(y))
-\delta\rho_\chi(x,y;\delta).\label{eq:badevent2}
\end{align}
For all pairs $x\in V(\gamma_{c_i})$ and $y\in V(\gamma_{c_j})$,
the event $X_{\max(i,j)}$ implies,
$$\sum_{k\in \mathbb Z}\|f_{k,c}(x)-f_{k,c}(y)\|_1\geq d_T(x,y)-(\eps_1+\eps_2)d_T(x,y)-\delta\rho_\chi(x,y;\delta).$$
In particular this shows that for $C=2\cdot C'+24$, if the events $X_1, X_2,\ldots, X_n$ all occur, then $\eqref{eq:15:claim}$ holds for all pairs $x,y\in V(T(c))$.
Hence we are left to show that
$$
\pr[X_1 \wedge \cdots \wedge X_n\mid\mathcal E] > 0\,.
$$

To this end, we define new events $\{ Y_i : i \in [n] \}$ and we show that for every $i \in [n]$,
\begin{equation}\label{def:yi}
\pr_{\mathcal E}\left[{X_1}\wedge \cdots \wedge {X_{i}} \mid {X_1}\wedge\cdots \wedge {X_{i-1}}\wedge { Y_i}\right]=1\,,
\end{equation}
and then we bound the probability that $Y_i$ does not occur by,
\begin{equation}\label{eq:yi:last}
\pr_{\mathcal E}\left[\overline{ Y_i}\right]\leq  2^{-3(\varphi(c_i)-\varphi(c))+1}\,.
\end{equation}
By, Lemma~\ref{lem:delta} and the definition of $f_{k,c}$ \eqref{def:fic}, we have  $\pr_{\mathcal E}[X_1]=1$. Since for all $i\in\{2,\ldots n\}$, $c_i\in \chi(E(T(c)))\setminus\{c\}$, we have
\begin{align*}
\pr_{\mathcal E}[{X_1}\wedge \cdots \wedge {X_n}]&\geq 1-\sum_{i=2}^n \pr_{\mathcal E}\left[\overline{Y_i}\right] \\
& \overset{\eqref{eq:yi:last}}{\geq}1-\sum_{i=2}^n2^{-3(\varphi(c_i)-\varphi(c))+1}\\&\overset{\eqref{cor:expsum}}{>} 1-2\cdot2^{(2-3)}={0},
\end{align*}
which completes the proof.

\medskip

For each $i \in [n]$, we define the event $Y_i$ as follows:  For all $j < i$, and all vertices $x \in R_c(c_i)$ and $y \in V(\gamma_{c_j})$ with $\mu(x,y)=\varphi(c)$,
we have
\begin{equation}
\sum_{k\in \mathbb Z}\|f_{k,c}(x)-f_{k,c}(y)\|_1-\sum_{k\in \mathbb Z}\|f_{k,c}(\Gamma_c(x))-f_{k,c}(y)\|_1
\geq\label{eq:goodevent}
\,(1-\eps_1/2)\,d_T(x,\Gamma_c(x))\,.
\end{equation}

We now complete the proof of Lemma \ref{lem:sec5:main:induction} by proving \eqref{def:yi} and \eqref{eq:yi:last}.

\medskip
\noindent
{\bf Proof of \eqref{def:yi}.} Suppose that $X_1,\ldots,X_{i-1}$ and $Y_i$ hold.
We will show that $X_i$ holds as well.
First note for all vertices in $x,y\in V(\gamma_{c_i})$, by Lemma~\ref{lem:delta} and the definition of $f_{k,c_i}$ \eqref{def:fic}, we have
$$d_T(x,y)=\sum_{k\in \mathbb Z}\|f_{k,c_i}(x)-f_{k,c_i}(y)\|_1=\sum_{k\in \mathbb Z}\|f_{k,c}(x)-f_{k,c}(y)\|_1,$$
thus we only need to prove \eqref{eq:badevent2} for pairs $x\in V(\gamma_{c_i})$, and $y\in V(\gamma_{c_j})$ for with $j< i$ and $\mu(x,y)=\varphi(c)$.
We now divide the pairs with one endpoint in $\gamma_{c_i}$ into two cases based on $\Gamma_c$.

\medskip
\noindent
{\bf Case I: $x\in V(\gamma_{c_i})$ with $x\neq \Gamma_c(x)$, and $y\in V(\gamma_{c_j})$ for some $j< i$, and $\mu(x,y)=\varphi(c)$.}

\medskip
In this case, by Lemma~\ref{obv:rcdelta}, there exists a vertex $z\in R_{c}(c_i)$ such that
$$d(x,v_{c_i}) \leq d(z,v_{c_i})\leq 2\max\left(\delta\, \len (E(\gamma_{c_i})),d_T(x,v_{c_i})\right).$$
If $d(x,v_{c_i})\leq \delta\,\len (E(\gamma_{c_i}))$, then by \eqref{eq:rho}, we have $\len( E(\gamma_{c_i})) = \rho_\chi(x,v_{c_i};\delta)$, hence
\begin{align}
d_T(z,\Gamma_c(z))&\leq d_T(v_{c_i},\Gamma_c(z))+2\max( \delta \,\len(E(\gamma_{c_i})),d_T(x,v_{c_i}))\nonumber\\
 &\leq d_T(v_{c_i},\Gamma_c(z))+2\max( \delta \,\rho_\chi(x,v_{c_i};\delta),d_T(x,v_{c_i}))\nonumber\\
&\leq d_T(v_{c_i},\Gamma_c(z))+ 2\,\delta \,\rho_\chi(x,v_{c_i};\delta)+2\,d_T(x,v_{c_i})\nonumber\\
&\leq 2\delta \,\rho_\chi(x,v_{c_i};\delta)+2\,d_T(x,\Gamma_c(z)).\label{eq:rcdelta:last}
\end{align}

\begin{figure}[htbp]
\begin{center}
  \includegraphics[width=9cm]{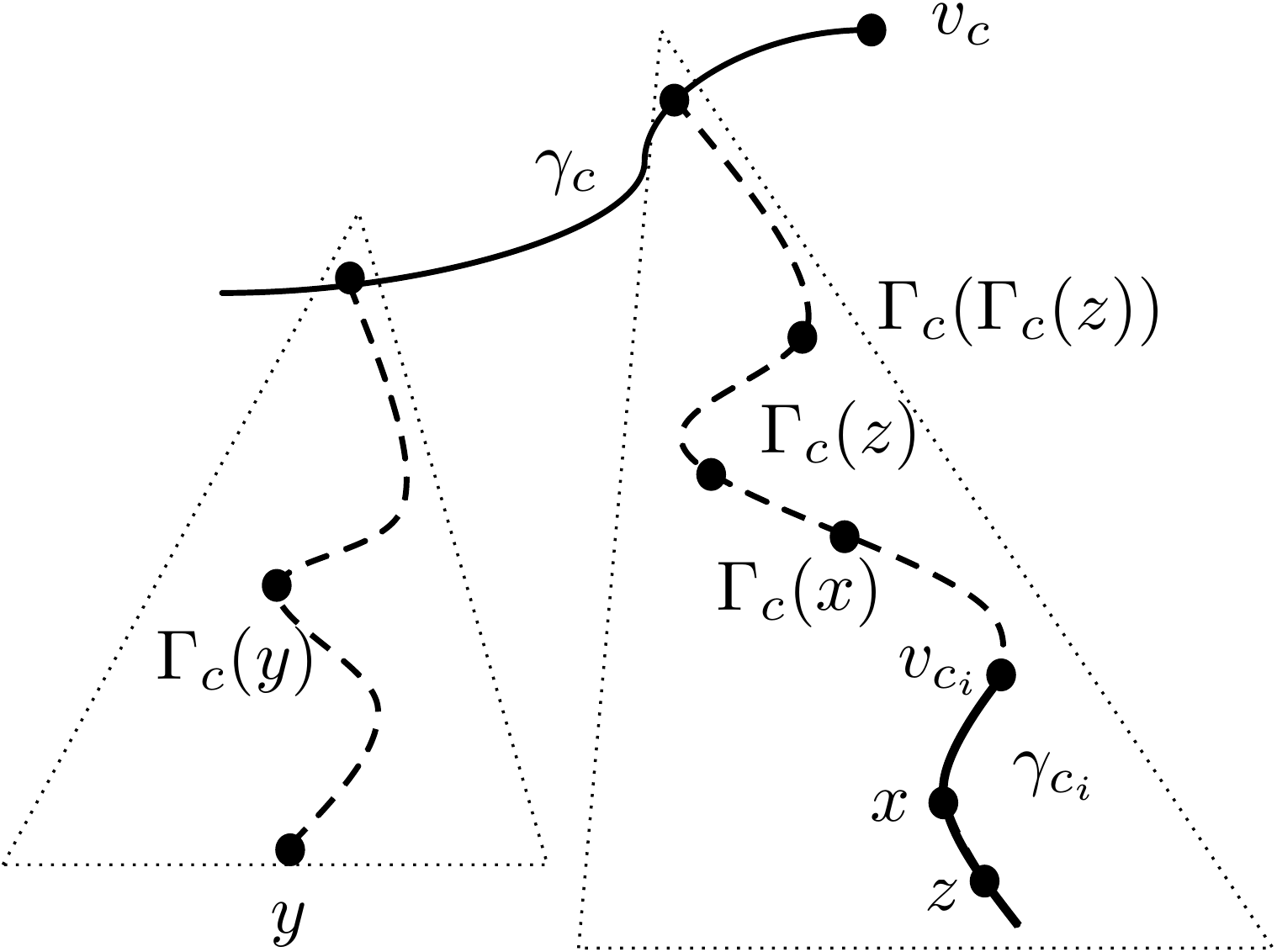}
  \caption{{Position of vertices in the subtree $T(c)$ for Case I.}}
  \label{fig:main:eq}
\end{center}
\end{figure}

Since $z \in R_c(c_i)$, by definition
we have $\Gamma_c(z)\neq z$, therefore by Lemma~\ref{obv:gamma1}, $\Gamma_c(z)=v_{c'}$ for some color $c'\in \chi(P_{z\,v_c})\setminus\{c\}$.
The function $\varphi$ is non-decreasing along any root leaf path, hence $\chi(\Gamma_c(z),p(\Gamma_c(z)))= c_{\ell}$ for some $\ell<i$.

We refer to Figure~\ref{fig:main:eq} for the relative position of the vertices referenced in the following inequalities.
Using our
assumption that $X_1,\ldots, X_{i-1}$ and $Y_i$ hold, we can write
\begin{align*}
\sum_{k\in \mathbb Z}\|f_{k,c}(z)-f_{k,c}(y)\|_1&\overset{Y_i}{\geq} d_T(\Gamma_c(z),z)-(\eps_1/2)\,d_T(z,\Gamma_c(z))+\sum_{k\in \mathbb Z}\|f_{k,c}(\Gamma_c(z))-f_{k,c}(y)\|_1\\
&\!\!\!\!\!\!\!\overset{X_{\max(\ell,j)}}{\geq} d_T(\Gamma_c(z),z)-(\eps_1/2)\,d_T(z,\Gamma_c(z))\\
&\qquad +d_T(\Gamma_c(z),y)-\eps_2\, d_T(\Gamma_c(\Gamma_c(z)),\Gamma_c(y))
-\eps_1\,d_T(\Gamma_c(z),y)
-\delta\,\rho_\chi(\Gamma_c(z),y;\delta)\\
&\geq d_T(y,z)-(\eps_1/2)\,d_T(z,\Gamma_c(z))\\
&\qquad -\eps_2\,d_T(\Gamma_c(\Gamma_c(z)),\Gamma_c(y))
-\eps_1\,d_T(\Gamma_c(z),y)
-\delta\,\rho_\chi(\Gamma_c(z),y;\delta)\,.
\end{align*}

\noindent
We may assume that $\e_1 < 1$, otherwise the statement of the lemma is vacuous.
Using the preceding inequality, and applying Lemma~\ref{lem:correlation} on pairs $(z,y)$ and $(x,y)$ implies that
\begin{align*}
\sum_{k\in \mathbb Z}\|f_{k,c}(x)-f_{k,c}(y)\|_1
&\geq d_T(x,y)-(\eps_1/2)\,d_T(z,\Gamma_c(z))\\
&\qquad -\eps_2\,d_T(\Gamma_c(\Gamma_c(z)),\Gamma_c(y))
-\eps_1\,d_T(\Gamma_c(z),y)
-\delta\,\rho_\chi(\Gamma_c(z),y;\delta)\\
&\overset{\eqref{eq:rcdelta:last}}{\geq}
d_T(x,y)-(\eps_1/2)\left(\vphantom{\bigoplus} 2\,d_T(x,\Gamma_c(z))+2\delta\, \rho_\chi(x,v_{c_i};\delta)\right)\\
&\qquad -\eps_2\,d_T(\Gamma_c(\Gamma_c(z)),\Gamma_c(y))
-\eps_1\,d_T(\Gamma_c(z),y)
-\delta\,\rho_\chi( \Gamma_c(z),y;\delta),
\end{align*}
where in the last line we have used the fact that $\eps_1 < 1$.

We have $\chi(x,p(x))=\chi(z,p(z))=c_i$. Moreover, since $\Gamma_c(z)\neq z$, using Lemma~\ref{obv:gamma1} it is easy to check that $x\in P_{z\,\Gamma_c(z)}$. Therefore, by Lemma~\ref{obv:gamma}, $ d_T(\Gamma(\Gamma_c(z)),y)\leq d_T(\Gamma_c(z),y)\leq  d_T(\Gamma_c(x),y),$ and combining this with the preceding inequality yields,
\begin{align*}
\sum_{k\in \mathbb Z}\|f_{k,c}(x)-f_{k,c}(y)\|_1
&\geq d_T(x,y)-(\eps_1/2)\left(\vphantom{\bigoplus} 2\,d_T(x,\Gamma_c(z))+2\delta\, \rho_\chi(x,v_{c_i};\delta)\right)\\
&\qquad -\eps_2\,d_T(\Gamma_c(x),\Gamma_c(y))
-\eps_1\,d_T(\Gamma_c(z),y)
-\delta\rho_\chi( \Gamma_c(z),y;\delta).
\end{align*}

Recall the definition of  $C(x,y;\delta)$ in \eqref{eq:rho}. Since by Lemma~\ref{obv:gamma1}, $\Gamma_c(z)=v_{c'}$ for some color $c'\in \chi(P_{z\,v_c})\setminus\{c\}$, we have
$C(\Gamma_c(z),y;\delta)\subseteq C(v_{c_i},y;\delta)$, hence $\rho_\chi(v_{c_i},y;\delta)\geq \rho_\chi(\Gamma_c(z),y;\delta)$ and thus,
\begin{align*}
\sum_{k\in \mathbb Z}\|f_{k,c}(x)-f_{k,c}(y)\|_1
&\geq d_T(x,y)-(\eps_1/2)\left(\vphantom{\bigoplus}2\,d_T(x,\Gamma_c(z))+2\delta \,\rho_\chi(x,v_{c_i};\delta)\right)\\
&\qquad -\eps_2\,d_T(\Gamma_c(x),\Gamma_c(y))
-\eps_1\, d_T(\Gamma_c(z),y)
-\delta\,\rho_\chi( v_{c_i},y;\delta) \\
&\geq d_T(x,y)-\eps_1\,d_T(x,\Gamma_c(z))
 -\eps_2\,d_T(\Gamma_c(x),\Gamma_c(y))
-\eps_1d_T(\Gamma_c(z),y)\\
&\qquad-\delta\big(\rho_\chi( v_{c_i},y;\delta)+\eps_1\rho_\chi( x,v_{c_i};\delta)\big)\\
&\geq d_T(x,y)-\eps_1\,d_T(x,\Gamma_c(z))
 -\eps_2\,d_T(\Gamma_c(x),\Gamma_c(y))
-\eps_1d_T(\Gamma_c(z),y)\\
&\qquad-\delta\big(\rho_\chi( x,v_{c_i};\delta)+\rho_\chi( v_{c_i},y;\delta)\big),
\end{align*}
where in the last line we have again used that $\eps_1 < 1$.

The set of colors that appear on the paths  $P_{x\,v_{c_i}}$ and $P_{v_{c_i} y}$ are disjoint, therefore $\rho_\chi( x,y;\delta)=\rho_\chi( x,v_{c_i};\delta)+\rho_\chi( v_{c_i},y;\delta)$,
and
\begin{align*}
\sum_{k\in \mathbb Z}\|f_{k,c}(x)-f_{k,c}(y)\|_1&\geq d_T(x,y)-\eps_1\,d_T(x,\Gamma_c(z))
 -\eps_2\,d_T(\Gamma_c(x),\Gamma_c(y))
-\eps_1\,d_T(\Gamma_c(z),y)
-\delta\rho_\chi( x,y;\delta)\\
&= d_T(x,y)
-\eps_1\,d_T(x,y)
 -\eps_2\, d_T(\Gamma_c(x),\Gamma_c(y))
-\delta\rho_\chi( x,y;\delta).
\end{align*}

\medskip
\noindent
{\bf Case II: $x\in V(\gamma_{c_i})$ with $x= \Gamma_c(x)$, and $y\in V(\gamma_{c_j})$ for some $j<i$, and $\mu(x,y)=\varphi(c)$.}

\medskip

In this case, if $x\in V(\gamma_c)$ then the event $X_j$ implies  \eqref{eq:badevent2}. On the other hand, suppose that $x\in V(T(c'))$ for some $c'\in \rho^{-1}(c)$.
Recall that ${\eps_2\over 2}=C'\eps$, where $C'$ is the constant from Lemma~\ref{lem:triangle}. By Lemma~\ref{lem:triangle} (with $c'$, $x$, and $\eps_2\over 2$ substituted for $c$, $v$, and $\eps$, respectively,
 in the statement of Lemma~\ref{lem:triangle}), there exist vertices $u,u'\in \{x\}\cup \{v_a:a\in\ \chi(E(P_{x\,v_{c'}}))\}$ such that
\begin{equation}\label{eq:xeps2}
d_T(x,u)\leq (\eps_2/2) \,d_T(u',u).
\end{equation}
and for all vertices $z\in V(P_{u'u})\setminus \{u'\}$ and for all $k \in \mathbb Z$,
\[
\tau_k(z) \neq 0 \implies
2^k < \left({{d_T(u,u')} \over \eps(\varphi(\chi(u,p(u)))-\varphi(\chi(v_{c'},p(v_{c'}))))}\right).
\]
We have $\chi(v_{c'},p(v_{c'}))=c$, and this condition is exactly the same condition as \eqref{def:gamma:last} for $\Gamma_c(u)$, therefore
\begin{equation}\label{eq:xeps21}
d_T(x,u)\leq (\eps_2/2)\, d_T(u',u)\leq (\eps_2/2)\, d_T(\Gamma_c(u),u).
\end{equation}
Note that, the assumption that $\Gamma_c(x)=x$ implies that, $u\neq x$ and $u=v_a$ for some $a\in \chi(E(P_{u,v_{c'}}))$.

We have,
\begin{align}
\sum_{k\in \mathbb Z}\|f_{k,c}(x)-f_{k,c}(y)\|_1-\sum_{k\in \mathbb Z}\|f_{k,c}(u)-f_{k,c}(y)\|_1
&\geq -\sum_{k\in \mathbb Z}\|f_{k,c}(x)-f_{k,c}(u)\|_1\nonumber\\
&\overset{\eqref{lem:rlpath}}{=}
-d_T(x,u)\nonumber\\
&\overset{\eqref{eq:xeps21}}{\geq} d_T(x,u)-\eps_2\,  d_T(u,\Gamma_c(u))\nonumber\\
&\geq d_T(x,u)-\eps_2\, d_T(x,\Gamma_c(u))\nonumber\\
&= d_T(x,u)-\eps_2\, d_T(\Gamma_c(x),\Gamma_c(u)).\label{eq:15:x1}
\end{align}
Since $u=v_a$ for some $a\in \chi(E(P_{u,v_{c'}}))$, $\chi(u,p(u))=c_\ell$, for some $\ell<i$, and $X_{\max(\ell,j)}$ implies that,
\begin{align}
\sum_{k\in \mathbb Z}\|f_{k,c}(u)-f_{k,c}(y)\|_1&\geq d_T(u,y)-
\eps_2\, d_T(\Gamma_c(u),\Gamma_c(y))
-\eps_1\, d_T(u,y)
-\delta\,\rho_\chi(u,y;\delta).\nonumber
\end{align}
Recall the definition of  $C(x,y;\delta)$ in \eqref{eq:rho}, We have $u=v_a$ for some $a\in \chi(E(P_{u,v_{c'}}))$, therefore $C(u,y;\delta)\subseteq C(x,y;\delta)$, and $\rho_\chi(u,y;\delta)\leq \rho_\chi(x,y;\delta)$. \
Now we can write,
\begin{align}
\sum_{k\in \mathbb Z}\|f_{k,c}(u)-f_{k,c}(y)\|_1&\geq d_T(u,y)-
\eps_2\, d_T(\Gamma_c(u),\Gamma_c(y))
-\eps_1\, d_T(u,y)
-\delta\, \rho_\chi(x,y;\delta).
\label{eq:15:x2}
\end{align}
Adding \eqref{eq:15:x1} and \eqref{eq:15:x2} we can conclude that
\begin{align*}
\sum_{k\in \mathbb Z}\|f_{k,c}(x)-f_{k,c}(y)\|_1&\geq d_T(u,y)+d_T(u,x)
-\eps_2\, d_T(\Gamma_c(x),\Gamma_c(u))+d_T(\Gamma_c(u),\Gamma_c(y))\\
&\qquad-\eps_1\,d_T(x,y)
-\delta\,\rho_\chi(x,y;\delta)\\
&\geq d_T(x,y)
-\eps_2 \, d_T(\Gamma_c(x),\Gamma_c(y))
-\eps_1\,d_T(x,y)
-\delta\,\rho_\chi(x,y;\delta),
\end{align*}
completing the proof of \eqref{def:yi}.

\medskip
\noindent
{\bf Proof of \eqref{eq:yi:last}.}
We prove this inequality by first bounding the probability that \eqref{eq:goodevent} holds for a fixed $x$ and all $y\in V(\gamma_{c_j})$ (for a fixed $j\in\{1,\ldots, i-1\}$) with $\mu(x,y)=\varphi(c)$.
Then we use a union bound to complete the proof.

We start the proof by giving some definitions.
For a vertex $x\in R_{c}(c_i)$, let
$$
S_x=\left\{\vphantom{\bigoplus}\textrm{$j\in\{1,\ldots,i-1\}$: there exists a $v\in V(\gamma_{c_j})$ such that $\mu(x,v)=\varphi(c)$}\right\}.
$$
And for $a\in S_x$, we define $w(x;a)$ as the vertex $v \in V(\gamma_a)$ which is furthest from the root among those satisfying $\mu(x,v)=\varphi(c)$.
Finally for $x\in R_{c}(c_i)$, we put
$$\beta_x=\max\left\{k\in \mathbb Z: \exists z\in P_{x\,\Gamma_c(x)}\setminus\{\Gamma_c(x)\},\,\tau_k(z)\neq 0\right\}\,.$$

Inequality \eqref{def:gamma:last} implies,
\begin{equation}\label{eq:betaphi}
2^{\beta_x}< {d_T(x,\Gamma_c(x))\over \eps(\varphi(c_i)-\varphi(c))}.
\end{equation}

By definition of $R_c$, for all elements $x\in R_c(c_i)$, we have $\Gamma_c(x)\neq x$. Moreover, by Lemma~\ref{obv:gamma1}, $\Gamma_c(x)=v_{c'}$ for some $c'\in \chi(E(P_{x\,v_c}))\setminus \{c\}$.
Now, for $x\in R_c(c_i)$ and $a\in S_x$ we apply Lemma~\ref{lem:con} with $\eps_1/2=12\eps$ to write
\begin{align}\nonumber
&\pr_{\mathcal E}\left[\exists y\in P_{w(x;a),v_c}:\sum_{k\in \mathbb Z}\|f_{k,c}(x)-f_{k,c}(y)\|_1\leq (1-\eps_1/2)d_T(x,\Gamma_c(x))+
\sum_{k\in \mathbb Z}\|f_{k,c}(y)-f_{k,c}(\Gamma_c(x))\|_1\right]\\
&\hspace{4.2in}\leq  {1\over \lceil \log_2 1/\delta \rceil}{\exp\left(-12\frac{d_T(x,\Gamma_c(x))}{2^{\beta_x+2}\eps}\right)}\,\nonumber\\
&\hspace{4.2in}\overset{\eqref{eq:betaphi}}{\leq}  \frac{\exp(-3(\varphi(c_i)-\varphi(c)))}{\lceil \log_2 1/\delta \rceil}\,.\label{eq:cic0}
\end{align}
Note that, for all $y\in V(\gamma_{c_a})$ with $\mu(x,y)=\varphi(c)$,  we have $y\in P_{w(x;a),v_c}$.

By definition of $R_c(c_i)$, $|R_{c}(c_i)|\leq \lceil\log_2\delta^{-1}\rceil$. We also have $\varphi(c_j)\leq \varphi(c_i)$ for $j<i$, and by Corollary~\ref{cor:count}, $|S_x| \leq i < 2^{\varphi{(c_i)}-\varphi(c)+1}$.
Taking a union bound over all $x\in R_c(c_i)$ and $a\in S_x$ implies,
\begin{align*}
\pr_{\mathcal E}[\overline{Y_i}]&\overset{\eqref{eq:cic0}}{\leq}
\sum_{x\in R_c(c_i)} |S_x|\left( {1\over \lceil\log_2\delta^{-1}\rceil} \exp(-3(\varphi(c_i)-\varphi(c)))\right)\\
&<
\left(\lceil\log_2\delta^{-1}\rceil 2^{\varphi{(c_i)}-\varphi(c)+1}\right)\left( {1\over \lceil\log_2\delta^{-1}\rceil} \exp(-3(\varphi(c_i)-\varphi(c)))\right)\\
&= 2^{\varphi{(c_i)}-\varphi(c)+1}  \exp(-3(\varphi(c_i)-\varphi(c)))\,.
\end{align*}
Since $\varphi(c_i)\geq\varphi(c)$, by an elementary calculation we conclude that
\[\pr_{\mathcal E}[\overline{Y_i}] < 2\cdot 2^{-3(\varphi(c_i)-\varphi(c))}\,,\]
which completes the proof of  \eqref{eq:yi:last}.
\end{proof}

Finally, we present the proof of Lemma~\ref{lem:sec5:main}.

\begin{proof}[Proof of Lemma~\ref{lem:sec5:main}]
Let $C$ be the same constant as the constant in  Lemma~\ref{lem:sec5:main:induction}. For the sake of contradiction, suppose that
\begin{equation*}
\pr\left[\forall x,y \in V,\,\, (1 -C\e) \,d_T(x,y) - \delta\, \rho_{\chi}(x,y;\delta) \leq \sum_{i\in \mathbb Z}\|f_i(x)-f_i(y)\|_1 \leq d_T(x,y)\right]=0\,.
\end{equation*}
Now let $c\in \chi(E)\cup\{\col(r,p(r))\}$ be a color with a maximal value of $\varphi(c)$ such that,
\begin{equation}\label{eq:finalcon}
\pr\left[\forall x,y \in V(T(c)),\,\,(1 -C\e) \,d_T(x,y) - \delta\, \rho_{\chi}(x,y;\delta) \leq \sum_{i\in \mathbb Z}\|f_{i,c}(x)-f_{i,c}(y)\|_1 \leq d_T(x,y)\right]=0\,.
\end{equation}

For $a\in \chi(E)$,  $\kappa(a)>0$. Hence, for all $c'\in \rho^{-1}(c),$ by \eqref{def:phi}, $\varphi(c')>\varphi(c)$, and by maximality of $c$,
for all $c'\in \rho^{-1}(c),$ we have
\begin{equation*}
\pr\left[x,y \in V(T(c')),\,\,(1 -C\e) \,d_T(x,y) - \delta\, \rho_{\chi}(x,y;\delta) \leq \sum_{i\in \mathbb Z}\|f_{i,c'}(x)-f_{i,c'}(y)\|_1 \leq d_T(x,y)\right] > 0\,.
\end{equation*}
But now applying Lemma~\ref{lem:sec5:main:induction} contradicts \eqref{eq:finalcon}, completing the proof.
\end{proof}

\bibliographystyle{alpha}
\bibliography{socg}

\end{document}